\pdfoutput=1

 \documentclass[oneside,reqno,a4paper]{amsart}

\usepackage[mathscr]{eucal}
\usepackage{amssymb}
\usepackage[usenames,dvipsnames]{xcolor} 
\usepackage[normalem]{ulem}
\usepackage{amsthm}
\usepackage{bbold}
\usepackage{comment}
\usepackage{enumitem}
\usepackage{amsmath}
\usepackage{wasysym}
\usepackage{tikz-cd}
\usepackage{bbm} 
\usepackage{etoolbox} 
\usepackage{microtype}
\usepackage{adjustbox}

\usepackage[unicode,pagebackref]{hyperref} 

\definecolor{dark-red}{rgb}{0.5,0.15,0.15}
\definecolor{dark-blue}{rgb}{0.15,0.15,0.6}
\definecolor{dark-green}{rgb}{0.15,0.6,0.15}
\hypersetup{
    colorlinks, linkcolor=Blue,
    citecolor=Blue, urlcolor=Blue
}

\renewcommand*{\backref}[1]{}
\renewcommand*{\backrefalt}[4]{%
  \ifcase #1 %
No citations.
  \or
(cit. on p. #2).%
  \else
(cit on pp. #2).%
  \fi%
}

\usepackage[nameinlink,capitalise,noabbrev,sort]{cleveref}
\newcommand{\euscr}[1]{\EuScript{#1}} 
\newcommand{\dcat}{\euscr{D}} 

\newtheorem{thmx}{Theorem}

\numberwithin{equation}{section}
\usepackage[all]{xy}
\xyoption{line}
\usepackage{graphicx}
\usepackage{mathtools}

\swapnumbers 

\newtheorem{Thm}[equation]{Theorem}
\newtheorem*{Thm*}{Theorem}
\newtheorem{Prop}[equation]{Proposition}
\newtheorem{Lem}[equation]{Lemma}
\newtheorem{Cor}[equation]{Corollary}

\theoremstyle{remark}
\newtheorem{Def}[equation]{Definition}

\newtheorem{Not}[equation]{Notation}
\newtheorem{Exa}[equation]{Example}

\newtheorem{Conv}[equation]{Convention}

\newtheorem{Rem}[equation]{Remark}

\tikzset{
    labelrotatebelow/.style={anchor=north, rotate=90, inner sep=1.0mm}
}
\tikzset{
    labelrotateabove/.style={anchor=south, rotate=90, inner sep=1.0mm}
}

\newcommand{\nc}{\newcommand}
\nc{\dmo}{\DeclareMathOperator}

\newcommand\rightthreearrow{%
        \mathrel{\vcenter{\mathsurround0pt
                \ialign{##\crcr
                        \noalign{\nointerlineskip}$\rightarrow$\crcr
                        \noalign{\nointerlineskip}$\rightarrow$\crcr
                        \noalign{\nointerlineskip}$\rightarrow$\crcr
                }%
        }}%
}
\nc{\Beren}[1]{{\color{MidnightBlue}#1}}
\nc{\Drew}[1]{{\color{OliveGreen}#1}}
\nc{\Tobi}[1]{{\color{Orange}#1}}
\nc{\Dout}[1]{\Drew{\sout{#1}}}
\nc{\Bout}[1]{\Beren{\sout{#1}}}
\nc{\Tout}[1]{\Tobi{\sout{#1}}}
\newcommand{\pics}{\mathfrak{pic}}

\usepackage{todonotes}
\usepackage{quiver} 

\nc{\overbar}[1]{\mkern 1.5mu\overline{\mkern-1.5mu#1\mkern-1.5mu}\mkern 1.5mu}

\let\oldtikzcd\tikzcd
\def\tikzcd{{\ifnum0=`}\fi\oldtikzcd}

\let\oldendtikzcd\endtikzcd
\def\endtikzcd{\oldendtikzcd\ifnum0=`{\fi}}

\nc{\kappaaux}{g}
\nc{\kappam}{{\kappaaux({\frak m})}}
\nc{\kappaP}{{\kappaaux(\cat P)}}
\nc{\kappaQ}{{\kappaaux(\cat Q)}}
\nc{\kappaCP}{{\kappaaux_{\cat C}(\cat P)}}
\nc{\kappaDP}{{\kappaaux_{\dcat}(\cat P)}}
\nc{\kappaCQ}{{\kappaaux_{\cat C}(\cat Q)}}
\nc{\kappaDQ}{{\kappaaux_{\dcat}(\cat Q)}}
\nc{\kappaphiB}{{\kappaaux(\phi(\cat B))}}
\nc{\kappaphiQ}{{\kappaaux(\varphi(\cat Q))}}

\dmo{\Sub}{Sub}
\dmo{\Fr}{Fr}
\dmo{\Aut}{Aut}
\dmo{\Gal}{Gal}
\nc{\SpEn}{\cat S_{E(n)}}
\nc{\SpEnf}{\cat S_n}
\nc{\Lcomp}{L^{\mathrm{com}}} 
\nc{\Ucomp}{U^{\mathrm{com}}}
\nc{\Loco}[1]{\Loc_{\otimes}\hspace{-0.3ex}\langle #1 \rangle}
\nc{\bbullet}{{\scriptscriptstyle\hspace{-1pt}\bullet}}
\nc{\bullett}{{\scriptscriptstyle\bullet}\hspace{-1pt}}
\nc{\LF}{L\hspace{-0.2ex}F}
\nc{\SpG}{\Sp_G}
\nc{\SpGn}{\Sp_{G,n}}
\nc{\EG}{\bbE_G}
\nc{\EH}{\bbE_H}
\nc{\DEG}{\Der(\EG)}
\nc{\DEH}{\Der(\EH)}
\nc{\DE}{\Der(\bbE)}
\nc{\Prst}{{\cat P}\mathrm{r^{st}}}
\nc{\Mack}[2]{\mathrm{Mack}_{#1}(#2)}
\nc{\SC}{S\cat C}
\dmo{\fin}{{fin}}
\dmo{\DM}{DM}
\dmo{\fp}{fp}
\nc{\DMQ}{\DM_Q}
\dmo{\DerKal}{DMack}
\dmo{\Der}{D}
\dmo{\DMot}{DMot}
\dmo{\rmH}{H}
\dmo{\piu}{\underline{\pi}}
\dmo{\Sphere}{\mathbb{S}}
\nc{\HA}{{\rmH \hspace{-0.2em}\bbA}}
\nc{\HZ}{{\rmH \hspace{-0.2em}\bbZ}}
\nc{\HZbar}{{\rmH \hspace{-0.2em}\underline{\bbZ}}}
\nc{\Fp}{{\bbF_{\hspace{-0.1em}p}}}
\nc{\HFp}{{\rmH \hspace{-0.15em}\bbF_{\hspace{-0.1em}p}}}
\nc{\DHZG}{\Der(\HZ_G)}
\nc{\DHZH}{\Der(\HZ_H)}
\nc{\DHZK}{\Der(\HZ_K)}
\nc{\DHZGN}{\Der(\HZ_{G/N})}
\nc{\DHZGG}{\Der(\HZ_{G/G})}
\nc{\DHZCp}{\Der(\HZ_{C_p})}
\nc{\DHZGprime}{\Der(\HZ_{G'})}
\nc{\DHZ}{\Der(\HZ)}
\nc{\frakp}{\mathfrak{p}}
\nc{\frakq}{\mathfrak{q}}
\nc{\Z}{\mathbb{Z}}
\nc{\SSG}{\text{sSet}_*^G}
\nc{\sSet}{\text{sSet}}
\nc{\htimes}{\widehat{\otimes}}
\newcommand{\leftrarrows}{\mathrel{\raise.75ex\hbox{\oalign{%
  $\scriptstyle\leftarrow$\cr
  \vrule width0pt height.5ex$\hfil\scriptstyle\relbar$\cr}}}}
\newcommand{\lrightarrows}{\mathrel{\raise.75ex\hbox{\oalign{%
  $\scriptstyle\relbar$\hfil\cr
  $\scriptstyle\vrule width0pt height.5ex\smash\rightarrow$\cr}}}}
\newcommand{\Rrelbar}{\mathrel{\raise.75ex\hbox{\oalign{%
  $\scriptstyle\relbar$\cr
  \vrule width0pt height.5ex$\scriptstyle\relbar$}}}}
\newcommand{\longleftrightarrows}{\leftrarrows\joinrel\Rrelbar\joinrel\lrightarrows}

\makeatletter
\def\leftrightarrowsfill@{\arrowfill@\leftrarrows\Rrelbar\lrightarrows}
\newcommand{\xleftrightarrows}[2][]{\ext@arrow 3399\leftrightarrowsfill@{#1}{#2}}
\makeatother

\dmo{\csupp}{csupp}
\dmo{\Con}{Conj}
\dmo{\Id}{Id}
\dmo{\Loc}{Loc}
\dmo{\rmK}{\textrm{\rm K}}
\dmo{\Spc}{Spc}
\dmo{\thick}{thick}
\nc{\thickt}[1]{\thick_\otimes\langle #1 \rangle}
\dmo{\cone}{cone}
\dmo{\End}{End}
\dmo{\Mor}{Mor}
\dmo{\Hom}{Hom}
\dmo{\id}{id}
\dmo{\incl}{incl}
\dmo{\Img}{Im}
\dmo{\im}{im}
\dmo{\Ker}{Ker}
\dmo{\ind}{ind}
\dmo{\CoInd}{coind}
\dmo{\res}{res}
\dmo{\infl}{infl}
\dmo{\triv}{triv}
\dmo{\Tel}{Tel} 
\dmo{\grMod}{grMod}%

\dmo{\opname}{op}
\dmo{\SH}{SH}
\dmo{\smallb}{b}
\dmo{\Spec}{Spec}
\dmo{\supp}{supp}
\dmo{\Supp}{Supp}
\nc{\SHc}{{\SH^c}}
\nc{\SHp}{{\SH_{(p)}}}
\nc{\SHcp}{{\SH^c_{(p)}}}
\nc{\SHG}{\SH(G)}
\nc{\SHGp}{\SH(G)_{(p)}}
\nc{\SHGc}{\SHG^c}
\nc{\SHGcp}{\SHG^c_{(p)}}
\nc{\quadtext}[1]{\quad\textrm{#1}\quad}
\nc{\qquadtext}[1]{\qquad\textrm{#1}\qquad}
\nc{\adj}{\dashv}
\nc{\adjto}{\rightleftarrows}
\nc{\bbL}{\mathbb{L}}
\nc{\bbA}{\mathbb{A}}
\nc{\bbE}{\mathbb{E}}
\nc{\bbN}{\mathbb{N}}
\nc{\bbQ}{\mathbb{Q}}
\nc{\bbZ}{\mathbb{Z}}
\nc{\bbF}{\mathbb{F}}
\nc{\cat}[1]{\mathscr{#1}}
\nc{\ie}{{\sl i.e.}, }
\nc{\into}{\mathop{\rightarrowtail}}
\nc{\inv}{^{-1}}
\nc{\isoto}{\mathop{\overset{\sim}\to}}
\nc{\isotoo}{\mathop{\overset{\sim}\too}}
\nc{\onto}{\mathop{\twoheadrightarrow}}
\nc{\too}{\mathop{\longrightarrow}\limits}
\nc{\mapstoo}{\longmapsto}
\nc{\adh}[1]{\overline{#1}}
\nc{\adhpt}[1]{\adh{\{#1\}}}
\nc{\aka}{{a.\,k.\,a.}\ }
\nc{\calF}{\mathcal{F}}
\nc{\eg}{{\sl e.\,g.}}
\nc{\Homcat}[1]{\Hom_{\cat #1}}
\nc{\hook}{\hookrightarrow}
\nc{\ideal}[1]{\langle #1\rangle}
\nc{\ihom}{{\underline{\hom}}}
\nc{\Mid}{\,\big|\,}
\nc{\MMod}{\,\text{-}\Mod}%
\nc{\op}{^{\opname}}
\nc{\oto}[1]{\overset{#1}\to}
\nc{\otoo}[1]{\overset{#1}{\,\too\,}}
\nc{\sminus}{\!\smallsetminus\!}
\nc{\poplus}[1]{^{\oplus #1}}%
\nc{\potimes}[1]{^{\otimes #1}}
\nc{\sbull}{{\scriptscriptstyle\bullet}}
\nc{\SET}[2]{\big\{\,#1\Mid#2\,\big\}}
\nc{\SpcK}{\Spc(\cat K)}
\nc{\then}{\Rightarrow}
\nc{\unit}{\mathbbm{1}}
\nc{\xra}{\xrightarrow}
\nc{\phigeom}[1]{\widetilde{\Phi}^{#1}}
\nc{\phigeomb}[1]{\Phi^{#1}}
\dmo{\Oname}{O}
\dmo{\proper}{proper}
\dmo{\lenormal}{\unlhd}
\dmo{\lnormal}{\lhd}
\nc{\normal}{\trianglelefteq}
\nc{\Op}{\Oname^p}
\nc{\Oq}{\Oname^q}
\dmo{\Sp}{Sp}
\dmo{\Ho}{Ho}
\dmo{\Fin}{Fin}
\dmo{\add}{add}
\dmo{\Pic}{Pic}
\dmo{\Fun}{Fun}
\dmo{\Ext}{Ext}
\dmo{\CAlg}{CAlg}
\dmo{\CMon}{CMon}
\dmo{\CC}{\cat C} 
\dmo{\DD}{\dcat}
\dmo{\OO}{\mathcal{O}}
\dmo{\Map}{Map}
\dmo{\Span}{Span}
\dmo{\N}{N}
\dmo{\Cat}{Cat}
\dmo{\colim}{colim}
\dmo{\hocolim}{hocolim}
\dmo{\Ch}{Ch}
\dmo{\A}{\mathbb{A}^{eff}}
\nc{\AGeff}{\mathbb{A}_G^{\mathrm{eff}}}
\nc{\BGeff}{\mathcal{B}_G^{\mathrm{eff}}}
\nc{\BG}{{\mathcal{B}_G}}
\nc{\NBGeff}{{\N}{\BGeff}}
\dmo{\Stable}{Stable}
\dmo{\Ab}{Ab}
\dmo{\GrAb}{GrAb}
\dmo{\Set}{Set}
\dmo{\ev}{ev}
\dmo{\Spcl}{Spcl}
\nc{\Funadd}{\Fun_{\add}}
\dmo{\proj}{proj}
\dmo{\cof}{cof}
\dmo{\Tot}{Tot}
\nc{\StableG}{\Stable_{(A,\Gamma)}}

\dmo{\Coideal}{Coideal}
\dmo{\gen}{gen}
\newcommand{\ComodE}{\euscr{C}omod_{E_{0}E}} 
\newcommand{\Comod}{\euscr{C}omod}
\newcommand{\QCoh}{\euscr{QC}oh}

\nc{\mT}{\kern-0.5em\mod\kern-0.1em\text{-}\cat{T}^c}
\nc{\MT}{\Mod\kern-0.1em\text{-}\cat{T}}
\newcounter{enum-resume-hack}
\newcommand{\Mod}{\euscr{M}od} 
\Crefname{Prop}{Proposition}{Propositions}
\Crefname{Thm}{Theorem}{Theorems}
\Crefname{Lem}{Lemma}{Lemmas}
\begin{document}


\title[Invertible objects in Franke's comodule categories]{Invertible objects in Franke's comodule categories}
\author{Drew Heard}
\address{Department of Mathematical Sciences, Norwegian University of Science and Technology, Trondheim}
\email{drew.k.heard@ntnu.no}
\begin{abstract}
We study the Picard group of Franke's category of quasi-periodic $E_0E$-comodules for $E$ a 2-periodic Landweber exact cohomology theory of height $n$ such as Morava $E$-theory, showing that for $2p-2 > n^2+n$, this group is infinite cyclic, generated by the suspension of the unit. This is analogous to, but independent of, the corresponding calculations by Hovey and Sadofsky in the $E$-local stable homotopy category. We also give a computation of the Picard group of $I_n$-complete quasi-periodic $E_0E$-comodules when $E$ is Morava $E$-theory, as studied by Barthel--Schlank--Stapleton for $2p-2 \ge n^2$ and $p-1 \nmid n$, and compare this to the Picard group of the $K(n)$-local stable homotopy category, showing that they agree up to extension.  
\end{abstract}
\date{\today}

\makeatletter
\patchcmd{\@setaddresses}{\indent}{\noindent}{}{}
\patchcmd{\@setaddresses}{\indent}{\noindent}{}{}
\patchcmd{\@setaddresses}{\indent}{\noindent}{}{}
\patchcmd{\@setaddresses}{\indent}{\noindent}{}{}
\makeatother
\maketitle
\section{Introduction}
By the pioneering work of Quillen and Morava, it is known that the stable homotopy category has a `height' stratification, corresponding to the height filtration of formal groups. For example, height 0 is essentially rational stable homotopy theory, and in this case this has an entirely algebraic model, as a consequence of work of Serre \cite{Serre1953Groupes}.

The chromatic convergence theorem of Hopkins and Ravenel \cite{Ravenel1992Nilpotence} shows that for any $p$-local finite spectrum $X$, there is an equivalence $X \simeq \lim L_nX$, where $L_nX$ is the Bousfield localization with respect to a $p$-local Landweber exact homology theory $E$ of height $n$. Via certain chromatic squares, the computation of $L_nX$ can be roughly reduced to a computation of $L_{K(n)}X$, the Bousfield localization with respect to Morava $K$-theory $K(n)$. Correspondingly, the study of the $E$-local and $K(n)$-local categories, denoted $\Sp_n$ and $\Sp_{K(n)}$ respectively, are of great interest to homotopy theorists, and have been studied in detail in \cite{HoveyStrickl1999Morava}. 

One observation is that when the prime $p$ is large compared to $n$, the categories $\Sp_n$ and $\Sp_{K(n)}$ simplify, and become more algebraic. For example, in both cases the Adams spectral sequence computing $\pi_*(L_nS^0)$ or $\pi_*(L_{K(n)}S^0)$ collapses for $p \gg n$. Additionally, the Picard groups of the two categories become completely algebraic when $p \gg n$. Inspired by unpublished work of Franke \cite{franke1996uniqueness}, we have the following two results that make this asymptotic algebraic behavior more precise. 
\begin{Thm*} [Barthel--Schlank--Stapleton \cite{BarthelSchlankStapleton2020Chromatic}] 
For any non-principal ultrafilter $\mathcal{F}$ on the set of primes, there is an equivalence of symmetric monoidal stable $\infty$-categories
    \[
\prod_{\mathcal F}\Sp_n \simeq \prod_{\mathcal F} \dcat^{per}(\Comod_{E_0E}). 
    \]
\end{Thm*}
\begin{Thm*}[Pstr{\c a}gowski \cite{piotr_paper}, Patchkoria--Pstr{\c a}gowski \cite{patchkoria2021adams}]
      For $2p-2 > n^2+n$ and $k = 2p-2-n^2-n$ there is an equivalence \[
h_k\Sp_n \simeq h_k\dcat^{per}(\Comod_{E_0E})
    \]
    between the homotopy $k$-categories of the $\infty$-category of $E$-local spectra and the category of quasi-periodic $E_0E$-comodules.
\end{Thm*}
Here $\dcat^{per}(\ComodE)$ denotes the $\infty$-category of quasi-periodic $E_0E$-comodules for a 2-periodic Landweber exact cohomology theory $E$ of height $n$ (essentially) first introduced by Franke \cite{franke1996uniqueness}. Alternatively, this is the derived category of differential $E_*E$-comodules, where a differential $E_*E$-comodule is a pair $(M,d)$ consisting of an $E_*E$-comodule $M$ and $d \colon M \to M$ is a map of comodules of degree 1 satisfying $d^2 = 0$ (this is the approach taken in \cite{piotr_paper}, and the equivalence of the two approaches is given in \cite[Proposition 3.3]{piotr_paper}). 

The passage from the usual derived category $\dcat(\ComodE)$ to the quasi-periodic derived category $\dcat^{per}(\ComodE)$ is part of a more general construction due to Franke \cite{franke1996uniqueness}, expanded upon by Barnes and Roitzheim \cite{BarnesRoitzheim2011Monoidality}. We explain here a special case when restricted to the derived category of a suitable Hopf algebroid $(A,\Gamma)$ (note that by taking the Hopf algebroid $(A,A)$ this also includes the case of the derived category of a commutative ring). Given an invertible $\Gamma$-comodule $L$ and an integer $N \ge 0$, the derived category $\dcat^{(L,N)}(\Comod_{\Gamma})$ is obtained by considering the class of complexes for which there is a specified isomorphism $\alpha \colon L \otimes X \xrightarrow{\sim} X[N]$. For even periodic cohomology theories such as Morava $E$-theory, the quasi-periodic derived category $\dcat^{per}(\ComodE)$ is defined as $\dcat^{(L,2)}(\ComodE)$ for $L = E_2$, the invertible class in degree 2. 

The key result to the algebraicity theorem of Barthel--Schlank--Stapleton is the observation that for large primes $\dcat^{per}(\ComodE)$ satisfies a good theory of descent. Inspired by this, we introduce the notion of a descendable Hopf algebroid (\Cref{def:descendable_ha}). As expected, the associated quasi-periodic category satisfies a good theory of descent. In \Cref{thm:descent_hopf} we prove the following. 

\begin{thmx}\label{thm:descent_hopf_intro}
        Suppose $(A,\Gamma)$ is a descendable Hopf algebroid, and $L$ is an invertible $\Gamma$-comodule, then there is an equivalence of symmetric monoidal stable $\infty$-categories between $\dcat^{(L,N)}(\Comod_{\Gamma})$ and
\[
\Tot\left(\!\begin{tikzcd}
    {\dcat^{(L,N)}(\Mod_{A})} & {\dcat^{(\Gamma \otimes L,N)}(\Mod_{\Gamma})} & {\dcat^{(\Gamma \otimes \Gamma \otimes L,N)}(\Mod_{\Gamma \otimes \Gamma})}  \cdots
    \arrow[shift left=1, from=1-2, to=1-3]
    \arrow[shift right=1, from=1-1, to=1-2]
    \arrow[shift left=1, from=1-1, to=1-2]
    \arrow[shift right=1, from=1-2, to=1-3]
    \arrow[from=1-2, to=1-3]
\end{tikzcd} \right)
\]
\end{thmx}
In the case where $(A,\Gamma)$ is the Hopf algebroid associated to an even-periodic Landweber exact cohomology theory $E$ of height $n$, this recovers a result of Barthel--Schlank--Stapleton, although our methods are somewhat different - the precise translation of the two results is given in \Cref{rem:bss_descent_compare}. The extra generality we work in is not for no reason however - in \Cref{thm:descent_mrw} we show how it also applies to a certain Hopf algebroid that appears in the change of rings theorem of Miller--Ravenel \cite{MillerRavenel1977Morava}. 

As demonstrated in work of Mathew--Stojanoska \cite{MathewStojanoska2016Picard} descent is a powerful technique for determining the group of invertible objects (the Picard group) of a symmetric monoidal category. More precisely, \Cref{thm:descent_hopf_intro} implies a description of the Picard \emph{spectrum} (see \Cref{rem:picard_spectrum_properties}) associated to the category $\dcat^{(L,N)}(\Comod_{\Gamma})$: there is an equivalence of connective spectra 
\[
\resizebox{\columnwidth}{!}{$\displaystyle
\pics(\dcat^{(L,N)}(\Comod_{\Gamma}))\simeq \tau_{\ge 0}\Tot\left(
\begin{tikzcd}[ ampersand replacement=\&, column sep=1em]
{\pics(\dcat^{(L,N)}(\Mod_{A}))}\&{\pics(\dcat^{(\Gamma \otimes L,N)}(\Mod_{\Gamma}))} \& \cdots 
    \arrow[shift left=1, from=1-2, to=1-3]
    \arrow[shift right=1, from=1-1, to=1-2]
    \arrow[shift left=1, from=1-1, to=1-2]
    \arrow[shift right=1, from=1-2, to=1-3]
    \arrow[from=1-2, to=1-3]
\end{tikzcd} \right)
$}
\]
Studying the associated Bousfield--Kan spectral sequence in the case where $(A,\Gamma)$ is the Hopf algebroid associated to an even-periodic Landweber exact cohomology theory $E$ of height $n$, we prove the following in \Cref{thm:hovey_strickland_local}. 
\begin{thmx} Let $E$ be an even-periodic Landweber exact cohomology theory $E$ of height $n$. Suppose $2p-2 > n^2+n$, then $\Pic(\dcat^{per}(\Comod_{E_0E})) \cong \mathbb{Z}$, generated by the suspension of the unit.  
\end{thmx}
In the case $n = 1,p > 2$ this is a theorem of Barnes--Roitzheim \cite{BarnesRoitzheim2011Monoidality}, while for $p \gg n$ it is a consequence of work of Barthel--Schlank--Stapleton, see page 5 of \cite{BarthelSchlankStapleton2020Chromatic} (more precisely, it holds away from a finite set of primes). Both of these results, however, rely on the corresponding computation of the $E(n)$-local Picard group. The result we give is independent of this result (of course, the proof is similar in spirit). In \Cref{thm:pic_mrw} we also give a computation of the Picard group of the derived category of quasi-periodic comodules associated to the Hopf algebroid used by Miller--Ravenel; in particular, when $2p-2 \ge n^2$ and $p-1 \nmid n$, we show that this Picard group is cyclic of order $\mathbb{Z}/(2(p^n-1))$. 

Barthel--Schlank--Stapleton have also given a $K(n)$-local version of their algebraicity result, as follows. 
\begin{Thm*}[Barthel--Schlank--Stapleton \cite{BarthelSchlankStapleton2021Monochromatic}]
  For any non-principal ultrafilter $\mathcal{F}$ on the set of primes, there is an equivalence of symmetric monoidal stable $\infty$-categories
      \[
\prod_{\mathcal F}\Sp_{K(n)} \simeq \prod_{\mathcal F} \dcat^{per}(\Comod_{E_0E})^{\wedge}_{I_n}. 
    \]
\end{Thm*}
Here the algebraic model of the $K(n)$-local category appearing on the right-hand side is a certain Bousfield localization of $\dcat^{per}(\Comod_{E_0E})$, see \Cref{def:complete_categories} for a precise definition. 

In \Cref{def:complete_comodules} we therefore take up the study of the Picard group of the category $\dcat^{per}(\Comod_{E_0E})^{\wedge}_{I_n}$. We note that for technical reasons we fix a particular model for an even-periodic Landweber exact cohomology theory $E$ of height $n$, namely we use this to mean the Lubin--Tate theory associated to the Honda formal group law over $\bbF_{p^n}$. 

We first show that this category has a good theory of descent, and deduce a spectral sequence for computing the Picard spectrum of $\dcat^{per}(\Comod_{E_0E})^{\wedge}_{I_n}$, see \Cref{thm:picss_complete,cor:kn_ss}. Studying this spectral sequence, we prove the following in \Cref{thm:main_pic_local,thm:height1_kn,thm:pickn_ht2} (here $\mathbb{G}_n$ denotes the Morava stabilizer group). 
\begin{thmx}
  Suppose that $2p-2 \ge n^2$ and $(p-1) \nmid n$, then there is a short exact sequence 
  \[
0 \to H^1(\mathbb{G}_n,E_0^{\times}) \to \Pic(\dcat^{per}(\ComodE)^{\wedge}_{I_n}) \to \mathbb{Z}/2 \to 0.
\]
    When $n = 1$ and $p > 2$
    \[
\Pic(\dcat^{per}(\ComodE)^{\wedge}_{I_1}) \cong \mathbb{Z}_p \times \mathbb{Z}/2(p-1). 
    \]
        When $n = 2$ and $p > 3$, then 
    \[
\Pic(\dcat^{per}(\ComodE)^{\wedge}_{I_2}) \cong \mathbb{Z}_p^2 \times \mathbb{Z}/2(p^2-1). 
    \]
\end{thmx}
Up to extension, this identifies the Picard group of the $K(n)$-local category (see \cite{HopkinsMahowaldSadofsky1994Constructions}) with the Picard group of $\dcat^{per}(\Comod_{E_0E})^{\wedge}_{I_n}$. Once again, we note that the algebraicity results of Barthel--Schlank--Stapleton identify the two Picard groups away from a finite set of primes; our results are independent of Picard group computations in the $K(n)$-local category.\footnote{Aalthough in the cases $n = 1$ and $n = 2$, we do rely on computations in group cohomology that are also used in the $K(1)$ and $K(2)$-local computations.} 

  \subsection*{Acknowledgements}
  We thank Eric Peterson and Piotr Pstr{\c a}gowski for helpful conversations, and the referee for many helpful comments. The author is supported by grant number TMS2020TMT02 from the Trond Mohn Foundation.

\section{Quasi-periodic complexes of comodules}
Let $\cat A$ be a symmetric monoidal Grothendieck abelian category. Following work of Franke \cite{franke1996uniqueness}, Barnes and Roitzheim \cite{BarnesRoitzheim2011Monoidality} construct a category of quasi-periodic chain complexes. The construction of Barnes and Roitzheim relies on the choice of a self-equivalence $T \colon \cat A \to \cat A$ and a period $N$. We will only consider the special case where $T = -\otimes L$ is the tensor product with an invertible object $L$. 
\begin{Def}
	The category of quasi-periodic chain complexes $\Ch^{(L,N)}(\cat A)$ has objects the class of chain complexes in $\cat A$ which have a specified isomorphism $\alpha \colon L \otimes X \xrightarrow{\sim} X[N]$. A morphism is a chain map which commutes with the given isomorphisms. 
\end{Def}
\begin{Exa}
	Taking $L$ to be the tensor unit and $n = 0$, we see that $\Ch^{(\text{id},0)}(\cat A) \simeq \Ch(\cat A)$. 
\end{Exa}
\begin{Rem}\label{rem:periodification}
	By \cite[Lemma 1.2]{BarnesRoitzheim2011Monoidality} there is an adjunction 
	\[
P_{\cat A} \colon \Ch(\cat A) \leftrightarrows \Ch^{(L,N)}(\cat A) \colon U_{\cat A}
	\]
	where $U_{\cat A}$ is the forgetful functor, and the left adjoint $P_{\cat A}$ is the periodification functor defined on objects by 
	\[
P_{\cat A}(M) = \bigoplus_{k \in \mathbb{Z}} (M \otimes L^{\otimes k}[-kN]),
	\]
	i.e., it is the complex which has degree $t$ part $(P_{\cat A}(M))_t = \bigoplus_{k \in \mathbb{Z}} M_{t+kN} \otimes L^{\otimes k}$.
The differential on the summand $M_{t+kN} \otimes L^{\otimes k}$ is given by 
\[
(-1)^{kN}  d_{t+kN} \otimes \id_{L^{\otimes k}} \colon M_{t+kN} \otimes L^{\otimes k}    \to M_{t+kN-1} \otimes L^{\otimes k}.
\]

The forgetful functor also has a right adjoint, given by replacing the direct sum with the direct product in the definition of the periodification functor, see the remark before Proposition 1.3 of \cite{BarnesRoitzheim2011Monoidality}. 
\end{Rem}
\begin{Rem}\label{rem:qp_monoidal}
	From the definition we see that 
	\[
P_{\cat A}(M) \simeq P_{\cat A}(\unit) \otimes M. 
	\]
	where the tensor product is of chain complexes (we omit the forgetful functor). 
\end{Rem}
\begin{Rem}\label{rem:colimit_preserving_periodisation}
		Suppose that $F \colon \Ch(\cat A) \to \Ch(\cat B)$ is a symmetric monoidal colimit preserving functor with right adjoint $G$. Suppose that $L$ is invertible in $\Ch(\cat A)$, and note then that $F(L)$ is invertible in $\Ch(\cat B)$. It follows from the formula for periodification that if $M \in \Ch^{(L,N)}(\cat A)$, then $F(M) \in \Ch^{(F(L),N)}(\cat A)$, and moreover $ P_{\cat B} \circ F(R) \simeq F \circ P_{\cat A}(N)$ for $R \in \Ch(\cat A)$, i.e., the following diagram commutes:
\[\begin{tikzcd}
	{\Ch(\cat A)} & {\Ch(\cat B)} \\
	{\Ch^{(L,N)}(\cat A)} & {\Ch^{(F(L),N)}(\cat B)}
	\arrow["F", from=1-1, to=1-2]
	\arrow["F"', from=2-1, to=2-2]
	\arrow["{P_{\cat A}}"', from=1-1, to=2-1]
	\arrow["{P_{\cat B}}", from=1-2, to=2-2]
\end{tikzcd}\]
Taking right adjoints, it follows that for any $S \in \Ch^{(F(L),N)}(\cat B)$ we have $G \circ U_{\cat B}(S) \simeq U_{\cat A} \circ G(S)$. 
\end{Rem}
\begin{Rem}
We now specialize to the situation we are interested in. Let $(A,\Gamma)$ be a Hopf algebroid, always assumed to be an Adams Hopf algebroid \cite[Definition 1.4.3]{Hovey2004Homotopy}, i.e., $\Gamma$ is a filtered colimit of finitely-generated projective $A$-modules. This condition ensures that $\Gamma$ is a flat $A$-module via the left (or right) unit map\footnote{Such Hopf algebroids are called \emph{flat}.} (in fact, faithfully flat, as the unit maps are split by the counit map $\epsilon \colon \Gamma \to A$) and we let $\Comod_{\Gamma}$ be the symmetric monoidal abelian category of $(A,\Gamma)$-comodules, see \cite[Section 2]{Hovey2004Homotopy}. The condition that the Hopf algebroid is Adams ensures that the dualizable comodules generate $\Comod_{\Gamma}$ \cite[Proposition 1.4.4]{Hovey2004Homotopy}, and hence that $\Comod_{\Gamma}$ is Grothendieck abelian \cite[Proposition 1.4.1]{Hovey2004Homotopy}.

 Let $L$ be an invertible $\Gamma$-comodule, and consider the endofunctor of $\Comod_{\Gamma}$ given by tensoring with $L$. We then have the category $\Ch^{(L,N)}(\Comod_{\Gamma})$ of quasi-periodic complexes of $(A,\Gamma)$-comodules. Note that for the discrete Hopf algebroid $(A,A)$ we have $\Comod_{A} \simeq \Mod_A$, so this also includes as a special case the category $\Ch^{(L,N)}(\Mod_A)$ of  quasi-periodic complexes of $A$-modules. 
\end{Rem}
The following result is due to Barnes and Roitzheim \cite[Theorems 6.5 and 6.9]{BarnesRoitzheim2011Monoidality}
\begin{Thm}[Barnes--Roitzheim]\label{thm:barnes-roitzheim}
	Let $L$ be an invertible $\Gamma$-comodule, then there is a model structure on the category of quasi-periodic complexes of $\Gamma$-comodules $\Ch^{(L,N)}(\Comod_\Gamma)$, such that the resulting model category is cofibrantly-generated, proper, stable, and symmetric monoidal. Moreover, there is a symmetric monoidal Quillen adjunction 
	\[
P_{\cat A} \colon \Ch(\Comod_\Gamma) \leftrightarrows \Ch^{(L,N)}(\Comod_\Gamma) \colon U_{\cat A}
	\]
\end{Thm}
\begin{Not}\label{not:qp-tensor}
Following Barnes and Roitzheim, we denote the tensor product in $\Ch^{(L,N)}(\Comod_\Gamma)$ by $\otimes_{\mathcal P}$. 
\end{Not}
\begin{Rem}\label{rem:model_structures}
  The model structure used by Barnes and Roitzheim is called the quasi-projective model structure, because the weak equivalences are exactly the quasi-isomorphisms. This is a Bousfield localization of the relative projective model structure studied by Hovey in \cite[Section 2]{Hovey2004Homotopy}, see \cite[Corollary 6.4]{BarnesRoitzheim2011Monoidality}. We  note that in the case of the discrete Hopf algebroid $(A,A)$ the model structures are equivalent.  
\end{Rem}
\begin{Def}
	We let $\cat D(\Comod_{\Gamma})$ and $\cat D^{(L,N)}(\Comod_{\Gamma})$ denote the symmetric monoidal stable $\infty$-categories underlying these model categories (see \cite[Section 1.3.4]{lurie-higher-algebra}). 
\end{Def}
\begin{Rem}
We note that $\Comod_{\Gamma}$ is always locally presentable for an Adams Hopf algebroid as it is Grothendieck abelian. It follows that $\Ch(\Comod_{\Gamma})$ is Grothendieck abelian, and hence locally presentable as well. The adjunction mentioned in the last paragraph of \Cref{rem:periodification} can be used to show that $\Ch^{(L,N)}(\Comod_\Gamma)$ is locally presentable by applying \cite[Lemma 11.2]{Hashimoto2009Equivariant}. Alternatively, as noted after Corollary 2.4 of \cite{BarnesRoitzheim2011Monoidality}, $\Ch^{(L,N)}(\Comod_\Gamma)$ is the category of modules over the monad of the displayed adjunction in \Cref{rem:periodification}, and hence is locally presentable by \cite[Theorem 5.5.9]{MR1313497}. Along with \Cref{thm:barnes-roitzheim} we deduce that the Barnes--Roitzheim model structure is combinatorial. It then follows from \cite[Proposition 1.3.4.22]{lurie-higher-algebra} that the above $\infty$-categories are presentable. 
\end{Rem}
\begin{Rem}\label{rem:pu_adjunction}
	By \cite[Proposition 1.5.1]{Hinich2016DwyerKan} we then obtain a symmetric monoidal adjunction of stable $\infty$-categories 
		\[
P_{\Gamma} \colon \cat{D}(\Comod_{\Gamma}) \leftrightarrows \cat{D}^{(L,N)}(\Comod_{\Gamma}) \colon U_{\Gamma}
	\]
The adjunction therefore preserves commutative algebra objects; in particular, we have that $P_{\Gamma}(A) \in \CAlg(\cat D^{(L,N)}(\Comod_{\Gamma}))$. 

	We show that the monoidal Barr--Beck theorem (see \Cref{sec:monoida_barr}) holds for this adjunction, recovering \cite[Proposition 2.3]{BarnesRoitzheim2011Monoidality} in this case. 
\end{Rem}
\begin{Prop}[Barnes--Roitzheim]\label{prop:br_module}
	There is an equivalence of symmetric monoidal stable $\infty$-categories
\[
\cat D^{(L,N)}(\Comod_{\Gamma}) \simeq \Mod_{\cat{D}(\Comod_{\Gamma})}(P_{\Gamma}(A)). 
\]
\end{Prop}
\begin{proof}
		We must verify the conditions of \Cref{thm:mnn_bb} for the $(P_{\Gamma},U_{\Gamma})$ adjunction of \Cref{rem:pu_adjunction}. $U_{\Gamma}$ is conservative by construction, and commutes with colimits as it has a right adjoint (\Cref{rem:periodification}). The proof is completed if we can show that the projection formula holds, i.e., that 
		\[
U_{\Gamma}X \otimes Y \simeq U_{\Gamma}(X \otimes P_{\Gamma}(Y))
		\] 
for $X \in \dcat^{(L,N)}(\Comod_{\Gamma})$ and $Y \in \dcat(\Comod_{\Gamma})$. In fact, this holds for purely formal reasons, see \cite[Proposition 2.15]{BalmerDellAmbrogioSers2016GrothendieckNeeman}.
\end{proof}
\begin{Rem}
To be explicit: we have a commutative diagram as follows, where the diagonal arrows correspond to restriction and extension of scalars along the map $A \to P_{\Gamma}(A)$ in $\dcat(\Comod_{\Gamma})$:
\[\begin{tikzcd}
	{\dcat(\Comod_{\Gamma})} & {\dcat^{(L,N)}(\Comod_{\Gamma})} \\
	& {\Mod_{\dcat(\Comod_{\Gamma})}(P_{\Gamma}(A))}
	\arrow["P_{\Gamma}", shift left=1, from=1-1, to=1-2]
	\arrow["U_{\Gamma}", shift left=1, from=1-2, to=1-1]
	\arrow[shift right=1, from=2-2, to=1-1]
	\arrow["\sim"', shift right=1, from=1-2, to=2-2]
	\arrow[shift right=1, from=1-1, to=2-2]
	\arrow["\sim"', shift right=1, from=2-2, to=1-2]
\end{tikzcd}\]
\end{Rem}
\begin{Rem}
	As the forgetful functor from $\Gamma$-comodules to $A$-modules is symmetric monoidal, $L$ is also invertible in $\cat D(\Mod_A)$. We can therefore also form $\cat D^{(L,N)}(\Mod_A)$. Then, we have an equivalence of symmetric monoidal $\infty$-categories
	\[
\cat D^{(L,N)}(\Mod_A) \simeq \Mod_{\cat D(\Mod_A)}(P_{A}(A)). 
	\] 
\end{Rem}

\section{Descendable Hopf algebroids}
In this section we introduce the notion of a \emph{descendable} Hopf algebroid (\Cref{def:descendable_ha}). As the name suggests, we show (\Cref{thm:descent_hopf})  that if a Hopf algebroid is descendable, then $\dcat(\Comod_{\Gamma})$ (or more specifically, its associated category of quasi-periodic comodules) has a good theory of descent. 
\begin{Rem}
We recall that the forgetful functor $\epsilon_* \colon \cat D(\Comod_{\Gamma}) \to \cat D(\Mod_A)$ has a right adjoint, the extended (or cofree) comodule functor, defined by $\epsilon^*(M) = \Gamma \otimes_A M$, with structure map $\Delta \otimes M$, where $\Delta \colon \Gamma \to \Gamma \otimes \Gamma$ is the coproduct of the Hopf algebroid. We first show this extends to the derived category.  
\end{Rem}

\begin{Lem}\label{lem:free_forget}
	There is a symmetric monoidal adjunction  
\[
\epsilon_* \colon \dcat(\Comod_{\Gamma}) \leftrightarrows \dcat(\Mod_{A}) \colon \epsilon^*.
\]
The adjunction has the following properties:
\begin{enumerate}
	\item $\epsilon^*$ is conservative. 	
	\item $\epsilon^*$ commutes with colimits. 
	\item The adjunction satisfies the projection formula: the natural map
  \[
\epsilon^*(X) \otimes Y \to \epsilon^*(X \otimes \epsilon_*(Y))
  \]
  is an equivalence for all $X \in \dcat(\Mod_{A})$ and $Y \in \dcat(\Comod_{\Gamma}) $. 
\end{enumerate}
\end{Lem}
\begin{proof}
For the projective model structure, the existence of the adjunction is a special case of \cite[Proposition 2.2.1]{Hovey2004Homotopy}. We need to show that this is preserved by Bousfield localization at the homology isomorphisms (see \Cref{rem:model_structures}). $\epsilon_*$ clearly preserves cofibrations in the quasi-projective model structure, as they are the same as in the relative projective model structure. Moreover, as $\epsilon_*$ is conservative, it preserves quasi-isomorphisms, and so $\epsilon_*$ is a left Quillen functor.

The adjunction passes to underlying $\infty$-categories by \cite[Proposition 1.5.1]{Hinich2016DwyerKan}. We now verify the stated properties of the adjunction. For $(a)$, let $f \colon M \to N$ be a morphism in $\cat D(\Mod_A)$, with $\epsilon^*f = \Gamma \otimes f$ a quasi-isomorphism. Because $\Gamma$ is faithfully-flat, $f$ is a quasi-isomorphism as well, so that $\epsilon^*$ is conservative.  To see that $\epsilon^*$ commutes with colimits note that the $A$-module colimit of a diagram of comodules acquires the structure of a comodules and is in fact the colimit in $\Gamma$-comodule; the claim then follows because $\Gamma \otimes -$ commutes with colimits of $A$-modules. Finally, the projection formula can be proved in the usual way: using that both $\epsilon^*$ and $\epsilon_*$ preserve colimits, one reduces the claim to the case where $X = A$. This is then the claim that $\Gamma \otimes Y \simeq \Gamma \otimes (\epsilon_*Y)$ which holds by \cite[Lemma 1.1.5]{Hovey2004Homotopy} (note that $\Gamma$ is flat, so we do not need to derive the tensor product here). 
	\end{proof}
	\begin{Rem}
		Because $\epsilon_*$ is symmetric monoidal, $\epsilon^*$ is lax symmetric monoidal. We therefore obtain an adjunction at the level of commutative algebra objects \cite[Proposition 5.22]{MathewNaumannNoel2017Nilpotence}
		\[
\epsilon_* \colon \CAlg(\dcat(\Comod_{\Gamma})) \leftrightarrows \CAlg(\dcat(\Mod_A)) \colon \epsilon^*
		\]
		In particular, $\Gamma \simeq \epsilon^*(A)$ is a commutative algebra object in $\dcat(\Comod_{\Gamma})$. 
	\end{Rem}

\begin{Prop}\label{prop:free_forget_descendable}
		There is a symmetric monoidal adjunction, 
\[
\rho_* \colon \dcat^{(L,N)}(\Comod_{\Gamma}) \leftrightarrows \dcat^{(L,N)}(\Mod_{A}) \colon \rho^*
\]
compatible with the periodification functors, i.e., the diagram
\[\begin{tikzcd}
	{\dcat(\Comod_{\Gamma})} & {\dcat(\Mod_{A})} \\
	{\dcat^{(L,N)}(\Comod_{\Gamma})} & {\dcat^{(L,N)}(\Mod_{A})}
	\arrow["{\rho_*}", shift left=1, from=2-1, to=2-2]
	\arrow["{\rho^*}", shift left=1, from=2-2, to=2-1]
	\arrow["{\epsilon_*}", shift left=1, from=1-1, to=1-2]
	\arrow["{\epsilon^*}", shift left=1, from=1-2, to=1-1]
	\arrow["U_{\Gamma}"', shift right=1, from=2-1, to=1-1]
	\arrow["P_{\Gamma}"', shift right=1, from=1-1, to=2-1]
	\arrow["P_A"', shift right=1, from=1-2, to=2-2]
	\arrow["U_A"', shift right=1, from=2-2, to=1-2]
\end{tikzcd}\]
commutes. 

The adjunction has the following properties:
\begin{enumerate}
	\item $\rho^*$ is conservative. 	
	\item $\rho^*$ commutes with colimits. 
	\item The adjunction satisfies the projection formula. 
	\end{enumerate}
	Moreover, there is an equivalence 
	\[
\dcat^{(L,N)}(\Mod_{A}) \simeq \Mod_{\dcat^{(L,N)}(\Comod_{\Gamma})}(P_{\Gamma}(\Gamma)). 
	\]
\end{Prop}
\begin{proof}
	We first observe that $\epsilon_*(P_{\Gamma}(A)) \simeq P_A(A)$ (\Cref{rem:colimit_preserving_periodisation}). Applying \Cref{prop:bs_module,Rem:commutative_diagram_modules} with $T = P_{\Gamma}(A)$ to the adjunction of \Cref{lem:free_forget} we obtain an adjunction 
\[
\rho_* \colon  \Mod_{\dcat{(\Comod_{\Gamma})}}(P_{\Gamma}(A)) \longleftrightarrows \ \Mod_{\dcat(\Mod_{A})}(P_{A}(A)) \colon \rho^*
\]
satisfying the properties listed in the proposition.  Using \Cref{prop:br_module} we have equivalences 
	\[
	\begin{split}
\dcat^{(L,N)}(\Comod_{\Gamma}) &\simeq \Mod_{\dcat(\Comod_{\Gamma})}(P_{\Gamma}(A)) \\
 \dcat^{(L,N)}(\Mod_A) &\simeq \Mod_{\dcat(\Mod_{A})}(P_{A}(A))	
\end{split}\]
giving the adjunction and its properties.

For the `moreover' statement, we note that by \Cref{prop:bs_module}
\[
\dcat^{(L,N)}(\Mod_{A})  \simeq \Mod_{\dcat^{(L,N)}(\Comod_{\Gamma})}(\rho^*(P_{A}(A))). 
\]
But $\rho^*(P_{A}(A)) \simeq P_{\Gamma}(\epsilon^*(A)) \simeq P_{\Gamma}(\Gamma)$ (using the commutativity of the diagram in the proposition), and the result follows. 
\end{proof}

\begin{Prop}\label{prop:module_categories_descendable}
	For each $k \ge 2$ there is an equivalence of symmetric monoidal stable $\infty$-categories\footnote{Here, by abuse of notation we use $\otimes_{\cat P}$ for the appropriate symmetric monoidal product in both $\dcat^{(L,N)}(\Comod_{\Gamma})$ and $\dcat^{(L,N)}(\Mod_{A})$.}
	\[
	\begin{split}
\Mod_{\dcat^{(L,N)}(\Comod_{\Gamma})}(P_{\Gamma}({\Gamma})^{\otimes_{\cat P} k}) &\simeq \Mod_{\dcat^{(L,N)}(\Mod_A)}(P_A(\Gamma)^{\otimes_{\cat P} (k-1)}). 
\end{split}
	\]
\end{Prop}
\begin{proof}
Because $\rho_*$ is symmetric monoidal, we have 
\[
\rho_*(P_{\Gamma}(\Gamma)^{\otimes_{\cat P} (k-1)}) \simeq \rho_*(P_{\Gamma}(\Gamma))^{\otimes_{\cat P} {(k-1)}} \simeq P_A(\Gamma)^{\otimes_{\cat P} (k-1)}. 
\]

We apply \Cref{prop:bs_module} with $T = P_{\Gamma}(\Gamma)^{\otimes_{\cat P} (k-1)}$ to the adjunction
\[
\rho_* \colon \dcat^{(L,N)}(\Comod_{\Gamma}) \leftrightarrows \dcat^{(L,N)}(\Mod_{A}) \colon \rho^*
\]
of \Cref{prop:free_forget_descendable}. We deduce the existence of a commutative diagram
\[\begin{tikzcd}
	{\dcat^{(L,N)}(\Comod_{\Gamma})} & {\dcat^{(L,N)}(\Mod_{A})} \\
	\Mod_{{\dcat^{(L,N)}(\Comod_{\Gamma})}}(P_{\Gamma}(\Gamma)^{\otimes_{\cat P} (k-1)})& \Mod_{\dcat^{(L,N)}(\Mod_{A})}(P_A(\Gamma)^{\otimes_{\cat P} (k-1)})
	\arrow["{\zeta_*}", shift left=1, from=2-1, to=2-2]
	\arrow["{\zeta^*}", shift left=1, from=2-2, to=2-1]
	\arrow["{\rho_*}", shift left=1, from=1-1, to=1-2]
	\arrow["{\rho^*}", shift left=1, from=1-2, to=1-1]
	\arrow[""', shift right=1, from=2-1, to=1-1]
	\arrow[""', shift right=1, from=1-1, to=2-1]
	\arrow[""', shift right=1, from=1-2, to=2-2]
	\arrow[""', shift right=1, from=2-2, to=1-2]
\end{tikzcd}\]
and a symmetric monoidal equivalence
\[
\Mod_{\dcat^{(L,N)}(\Mod_{A})}(P_A(\Gamma)^{\otimes_{\cat P} (k-1)}) \simeq \Mod_{{\dcat^{(L,N)}(\Comod_{\Gamma})}}(\zeta^*(P_A(\Gamma)^{\otimes_{\cat P} (k-1)})).
\]
We recall that $\zeta^*$ is just given by $\rho^*$ after forgetting the module structure. We see then that 
\[
\zeta^*(P_A(\Gamma)^{\otimes_{\cat P} (k-1)}) \simeq \rho^*(P_A(\Gamma^{\otimes k-1})) \simeq P_{\Gamma}(\epsilon^*(\Gamma^{\otimes k-1})) \simeq P_{\Gamma}(\Gamma^{\otimes k}) \simeq P_{\Gamma}(\Gamma)^{\otimes_{\cat P} k}\]
 and the result follows. 
\end{proof}
\begin{Prop}\label{prop:descendable_module_Cats2}
	For each $k \ge 1$ there is an equivalence of symmetric monoidal stable $\infty$-categories 
	\[
\Mod_{\dcat^{(L,N)}(\Mod_A)}(P_A(\Gamma)^{\otimes_{\cat P} k}) \simeq \dcat^{(\Gamma^{\otimes k} \otimes L,N)}(\Mod_{\Gamma^{\otimes k}}). 
	\]
\end{Prop}
\begin{proof}
	By base change along $A \to \Gamma^{\otimes k}$ (given by using the left unit) we have an adjunction
	\[
\dcat(\Mod_{A}) \leftrightarrows \dcat(\Mod_{\Gamma^{\otimes k}})
	\]
	which satisfies the conditions of \Cref{thm:mnn_bb}. We can therefore apply \Cref{prop:bs_module} with $T = P_A(A)$ to deduce an induced adjunction 
	\[
\Mod_{\dcat(\Mod_{A})}(P_{A}(A)) \leftrightarrows \Mod_{\dcat(\Mod_{\Gamma^{\otimes k}})}(\Gamma^{\otimes k} \otimes P_A(A))
\]
and an equivalence 
\[
\Mod_{\dcat(\Mod_{\Gamma^{\otimes k}})}(\Gamma^{\otimes k} \otimes P_A(A)) \simeq \Mod_{\cat D(\Mod _A)}(\Gamma^{\otimes k} \otimes P_A(A))
\]
We now make two observations. The first is that there is an equivalence  $\Gamma^{\otimes k} \otimes P_A(A) \simeq P_A(\Gamma^{\otimes k}) \simeq P_A(\Gamma)^{\otimes_{\cat P} k}$ in $\cat D(\Mod_A)$ (\Cref{rem:qp_monoidal} and \Cref{not:qp-tensor}), so that 
\begin{equation}\label{eq:align_1}
\Mod_{\dcat(\Mod_{\Gamma^{\otimes k}})}(\Gamma^{\otimes k} \otimes P_A(A)) \simeq \Mod_{\cat D(\Mod _A)}(P_A(\Gamma)^{\otimes_{\cat P} k}).
\end{equation}
The second observation (which is a special case of \Cref{rem:colimit_preserving_periodisation} combined with \Cref{prop:br_module}) is that if we periodise $\cat D(\Mod_{\Gamma^{\otimes k}})$ with respect to $(\Gamma^{\otimes k} \otimes L,N)$ (note that $L$ invertible in $\Mod_A$ implies $\Gamma^{\otimes k} \otimes L$ is invertible in $Mod_{\Gamma^{\otimes k}}$), then 
\begin{align}\label{eq:align2}
\begin{split}
\cat D^{(\Gamma^{\otimes k} \otimes L,N)}(\Mod_{\Gamma^{\otimes k}}) &\simeq \Mod_{\dcat(\Mod_{\Gamma^{\otimes k}})}(P_{\Gamma^{\otimes k}}(\Gamma^{\otimes k}))\\& \simeq Mod_{\dcat(\Mod_{\Gamma^{\otimes k}})}(\Gamma^{\otimes k} \otimes P_{A}(A)) 
\end{split}
\end{align}
Finally, we note that because 
\[
\dcat^{(L,N)}(\Mod_A) \simeq \Mod_{\cat D(\Mod_A)}(P_{A}(A)). 
\]
we have
\begin{align}\label{eq:align3}
\begin{split}
\Mod_{\dcat^{(L,N)}(\Mod_A)}(P_A(\Gamma)^{\otimes_{\cat P} k})  & \simeq \Mod_{\Mod_{\cat D(\Mod_A)}(P_{A}(A))}(P_A(\Gamma)^{\otimes_{\cat P} k}) \\
&\simeq \Mod_{\dcat(\Mod_A)}(P_A(\Gamma)^{\otimes_{\cat P} k})
\end{split}
\end{align}
where the last equivalence follows by \cite[Corollary 3.4.1.9]{lurie-higher-algebra}. 

Together, we deduce that 
\begin{align*}
\Mod_{\dcat^{(L,N)}(\Mod_A)}(P_A(\Gamma)^{\otimes_{\cat P} k}) & \simeq \Mod_{\dcat(\Mod_{A})}(P_A(\Gamma)^{\otimes_{\cat P} k}) && [\text{\Cref{eq:align3}}]
\\
&\simeq \Mod_{\dcat(\Mod_{\Gamma^{\otimes_{\cat P} k}})}(\Gamma^{\otimes k} \otimes P_A(A))  && [\text{\Cref{eq:align_1}}]\\
& \simeq \cat D^{(\Gamma^{\otimes k} \otimes L,N)}(\Mod_{\Gamma^{\otimes k}})  && [\text{\Cref{eq:align2}}]
\end{align*}
as claimed.
\end{proof}
We now introduce a class of Hopf algebroids with a good theory of descent. 
\begin{Def}\label{def:descendable_ha}
  We say that a Hopf algebroid $(A,\Gamma)$ is descendable if the commutative algebra object $\Gamma \in \CAlg{(\dcat(\Comod_{\Gamma}))}$ is descendable in the sense of \cite[Definition 3.18]{Mathew2016Galois}; that is, the thick tensor ideal generated by $\Gamma$ is all of $\dcat(\Comod_{\Gamma})$. 
\end{Def}
\begin{Rem}\label{rem:descent_periodic}
  We note that if $\Gamma$ is descendable in $\dcat(\Comod_{\Gamma})$, then $P_{\Gamma}(\Gamma)$ is descendable in $\dcat^{(L,N)}(\Comod_{\Gamma})$ by \cite[Corollary 3.20]{Mathew2016Galois}. 
\end{Rem}
\begin{Thm}\label{thm:descent_hopf}
		Suppose $(A,\Gamma)$ is a descendable Hopf algebroid, and $L$ is an invertible $\Gamma$-comodule, then there is an equivalence of symmetric monoidal stable $\infty$-categories
\[
\dcat^{(L,N)}(\Comod_{\Gamma}) \simeq \Tot\left(\!\begin{tikzcd}
	{\dcat^{(L,N)}(\Mod_{A})} & {\dcat^{(\Gamma \otimes L,N)}(\Mod_{\Gamma})} & \cdots
	\arrow[shift left=1, from=1-2, to=1-3]
	\arrow[shift right=1, from=1-1, to=1-2]
	\arrow[shift left=1, from=1-1, to=1-2]
	\arrow[shift right=1, from=1-2, to=1-3]
	\arrow[from=1-2, to=1-3]
\end{tikzcd} \right)
\]
\end{Thm}

\begin{proof}
 By \Cref{rem:descent_periodic} and \cite[Proposition 3.22]{Mathew2016Galois} there is an equivalence 
	\[
	\begin{split}
\dcat^{(L,N)}(\Comod_{\Gamma}) \simeq \Tot(\Mod_{\dcat^{(L,N)}(\Comod_{\Gamma})}&(P_{\Gamma}(\Gamma)) \rightrightarrows \\&{\Mod_{\dcat^{(L,N)}(\Comod_{\Gamma})}(P_{\Gamma}(\Gamma)^{\otimes_{\cat P} 2})} \rightthreearrow \cdots ).
\end{split}
	\]	By \Cref{prop:free_forget_descendable,prop:module_categories_descendable,prop:descendable_module_Cats2} we have
	\[
	\begin{split}
{\Mod_{\dcat^{(L,N)}(\Comod_{\Gamma})}(P_{\Gamma}(\Gamma)^{\otimes_{\cat P} k})} &\simeq \Mod_{\dcat^{(L,N)}(\Mod_A)}(P_A(\Gamma)^{\otimes_{\cat P} (k-1)})\\
&\simeq \dcat^{(\Gamma^{\otimes (k-1)} \otimes L,N)}(\Mod_{\Gamma^{\otimes k-1}})
\end{split}
	\] 
	as required. 
\end{proof}
\begin{Rem}
    A similar, but simpler, argument shows that if $(A,\Gamma)$ is a descendable Hopf algebroid, then there is an equivalence of symmetric monoidal stable $\infty$-categories
\[
\dcat(\Comod_{\Gamma}) \simeq \Tot\left(\!\begin{tikzcd}
    {\dcat(\Mod_{A})} & {\dcat(\Mod_{\Gamma})} & \cdots
    \arrow[shift left=1, from=1-2, to=1-3]
    \arrow[shift right=1, from=1-1, to=1-2]
    \arrow[shift left=1, from=1-1, to=1-2]
    \arrow[shift right=1, from=1-2, to=1-3]
    \arrow[from=1-2, to=1-3]
\end{tikzcd} \right).
\]
This uses \Cref{lem:free_forget} and the non-periodic version of \Cref{prop:descendable_module_Cats2}. 
\end{Rem}
\section{Landweber exact Hopf algebroids and the moduli stack of formal groups}
In this section we recall the moduli stack of formal groups, and the relation between chromatic homotopy theory and the height filtration of this stack. We then introduce two examples of Hopf algebroids with descent. 
\begin{Not}
    Let $\mathcal{M}_{\mathrm{fg}}$ denote the moduli stack of ($p$-typical) formal groups. 
\end{Not}
\begin{Rem}
    For a detailed study of this moduli stack, we refer the reader to work of Naumann \cite{Naumann2007stack}, Goerss \cite{goerss_quasi-coherent_2008}, or the survey article \cite{bb}. For now, we simply recall the following results: the height filtration of formal groups gives rise to a filtration by closed substacks
    \[
\mathcal{M}_{\mathrm{fg}} \supset \mathcal{M}(1) \supset \mathcal{M}(2) \supset \cdots
    \]
    We let $\mathcal{M}_{\mathrm{fg}}^{\le n}$ denote the open complement of $\mathcal M(n+1)$, corresponding to formal groups of height at most $n$. We also set $\mathcal H(n) \coloneqq \mathcal M(n) \cap \mathcal{M}_{\mathrm{fg}}^{\le n}$, so that $H(n)$ is a locally closed substack corresponding to formal groups of height exactly $n$. 

    Finally, we note that there is an equivalence of symmetric monoidal $\infty$-categories $\QCoh(\mathcal{M}_{\mathrm{fg}}) \simeq \Comod_{BP_*BP}^{\mathrm{ev}}$ between quasi-coherent sheaves on the moduli stack of $p$-typical formal groups and evenly-graded $BP_*BP$-comodules (see, for example, \cite{Naumann2007stack}).  
\end{Rem}
We recall the following very general definition from \cite{HoveyStrickl2005Comodules}. We will generally use this in the case where $A = BP_*/I_h$, where $I_h = (v_0,\ldots,v_{h-1})$, with the usual conventions that $v_0 = p$ and $I_0 = (0)$. 
\begin{Def}[Hovey--Strickland]
	Let $(A,\Gamma)$ be a flat Hopf algebroid, and $f \colon A \to B$ a ring homomorphism. We say that $B$ is Landweber exact over $(A,\Gamma)$ (if it is clear, we will simply say Landweber exact) if the functor $M \mapsto M \otimes_A B$ from $\Gamma$-comodules to $B$-modules is exact. 
\end{Def}
\begin{Rem}\label{rem:corresponding_ha}
In this case, one can define a flat Hopf algebroid $(B,\Gamma_B)$ where $\Gamma_B = B \otimes_A \Gamma \otimes_A B$. See \cite[Section 2]{HoveyStrickl2005Comodules} for more details. 
\end{Rem}
\begin{Def}
    Given a morphism of rings $f \colon BP_*/I_h \to R$ we define the height of $f$ to be 
    \[
ht(f) \coloneqq \max \{ n \ge 0 \mid R/I_nR \ne 0 \}
    \]
    where we allow $\infty$ and set $ht(f) = 0$ in the case $R = 0$. 
\end{Def}
\begin{Rem}
	Our main examples will come from the following result of Naumann \cite[Proposition 28 and Corollary 30]{Naumann2007stack}. Note that in the following theorem, the category of comodules considered does not take into account any grading, see \cite[Remarks 29 and 34]{Naumann2007stack}.
\end{Rem}
\begin{Thm}[Naumann]\label{thm:naumann}
    Let $0 \le h \le n < \infty$, and let $BP_*/I_h \to R \ne 0$ be Landweber exact of height $n$.\footnote{Here we mean that $R$ is Landweber exact over the Hopf algebroid ($BP_*/I_h,BP_*BP/I_hBP_*BP$)} Let $(R,\Gamma) \coloneqq (R, R\otimes_{BP_*}BP_*BP \otimes_{BP_*}R)$ be the associated flat Hopf algebroid. Then, there is a symmetric monoidal equivalence of categories 
        \[\QCoh(\mathcal M(h) \cap \mathcal M_{fg}^{\le n}) \simeq \Comod_{\Gamma}
\]
    \end{Thm}
\begin{Exa}\label{ex:morava_e-theory}
	The first main example to keep in mind is the Lubin--Tate cohomology theory, which is the Landweber exact $BP_*$-algebra $E_{n,\ast}$, where 
	\[
E_{n,\ast} \cong \mathbb{W}(\mathbb{F}_{p^n})[u_1,\ldots,u_{n-1}][\![u^{\pm 1}]\!]. 
	\]

	The elements $u_i$ have degree 0, while $u$ has degree $2$. Here the map
	\[
BP_* \to E_{n,\ast}
	\]
	sends 
	\[
v_i \mapsto \begin{cases}
	u_i u^{2^i-1} & 1 \le i \le n-1 \\
	u^{2^n-1} & i=n \\
	0 & i>n.
\end{cases}
	\]
    This corresponds to the case of $h = 0$ above, and (taking into account the grading) gives an equivalence 
    \[
\QCoh(\mathcal M_{fg}^{\le n}) \simeq \Comod_{(E_n)_*(E_n)}^{ev}
    \]
    bewteen quasi-coherent sheaves on $\mathcal{M}_{fg}^{\le n}$ and evenly-graded $(E_n)_*(E_n)$-comodules. 
\end{Exa}
\begin{Exa}\label{exa:height_n}
    The map $BP_* \to {E_{n,_*}}$ in the previous example gives rise to a quotient map $BP_*/I_n \to E_*/I_n \cong \mathbb{F}_{p^n}[u^{\pm 1}]$ which is Landweber exact of height $n$. We denote the corresponding Hopf algebroid in the sense of \Cref{rem:corresponding_ha} as $(K_{n,\ast},\Sigma_{n,\ast})$. Naumann's theorem (in the case $h = n$) then gives rise to an equivalence 
        \[
\QCoh(\mathcal{H}(n)) \simeq \Comod_{\Sigma_{n,\ast}}^{ev}.
    \]
A direct proof of this is also given in \cite[Proposition 2.10]{barthel2021morava} where the Hopf algebroid $(K_{n,\ast},\Sigma_{n,\ast})$ is denoted by $(K_*,K_*E)$.

    Note that the Hopf algebroid $(K_{n,\ast},\Sigma_{n,\ast})$ is a 2-periodic version of the Hopf algebroid $(K(n)_*,K(n)_*K(n))$ used in \cite{MillerRavenel1977Morava}; Naumann's theorem implies that the comodule categories are equivalent in either case. For later use we recall that the Miller--Ravenel change of rings theorem \cite[Theorem 2.10]{MillerRavenel1977Morava} states that 
    \[
\Ext^{s,t}_{BP_*BP}(BP_*,v_n^{-1}BP_*/I_n) \cong \Ext^{s,t}_{\Sigma_{n,\ast}}(K_{n,\ast},K_{n,\ast}). 
    \]
    Morava has shown (see \cite[Theorem 6.2.10]{Ravenel1986Complex}) that if $p-1 \nmid n$ then the groups $\Ext^{s,t}_{\Sigma_{n,\ast}}(K_{n,\ast},K_{n,\ast}) = 0$ for $s > n^2$. Moreover, by Morava's change of rings theorem (see \cite[Theorem 6.5]{Devinatz1995Moravas} or \cite[Corollary 5.5]{barthel2021morava}) we have
    \[
\Ext^{s,t}_{\Sigma_{n,\ast}}(K_{n,\ast},K_{n,\ast}) \cong H^s(\mathbb{G}_n,(E_n)_t/(I_n)),
    \]
    where $\mathbb{G}_n \coloneqq \mathbb{S}_n \rtimes \Gal(\bbF_{p^n}/\bbF_{p})$ is the (extended) Morava stabilizer group, and $\mathbb{S}_n \coloneqq \Aut(H_n)$, for $H_n$ the Honda formal group of height $n$ over $\mathbb{F}_{p^n}$. The action of $\mathbb{G}_n$ on $(E_n)_*/I_n \cong \mathbb{F}_{p^n}[u^{\pm 1}]$ is described as follows: the Galois group acts in the usual way on $\mathbb{F}_{p^n}$, and $\mathbb{S}_n$ acts on $u$ by $s \cdot u = p(s)u$ where $p(s)$ takes the leading coefficient of $s \in \mathbb{S}_n$, which is necessarily an element of $\mathbb{F}_{p^n}^{\times}$. 
\end{Exa}
In both examples above, the base ring is even-periodic in the following sense. 
\begin{Def}
	Let $R$ be a graded ring, then $R$ is said to be even-periodic if it is of the form $R_0[u^{\pm 1}]$ where $u$ is an element of degree 2. 
\end{Def}
\begin{Rem}
	If particular, we see that $L \coloneqq R_2$ is an invertible $R$-module, and $R_{2n} \cong L^{\otimes n}$ for all $n$. 
\end{Rem}
\begin{Rem}
	 Let $0 \le h \le n < \infty$, and let $BP_*/I_h \to R \ne 0$ be Landweber exact of height $n$ for $R$ even periodic. Let $(R,\Gamma) \coloneqq (R, R\otimes_{BP_*}BP_*BP \otimes_{BP_*}R)$ be the associated flat Hopf algebroid, and $(R_0,\Gamma_0)$ the corresponding Hopf algebroid using degree 0 elements. Even-periodicity implies that there is an equivalence between $\Gamma_0$-comodules and even graded $\Gamma$-comodules, see \cite[Remark 3.14]{goerss_quasi-coherent_2008}. Along with \Cref{thm:naumann}, we see that
  \[
\Comod^{ev}_{\Gamma_0} \simeq \QCoh(\mathcal{M}(h) \cap \mathcal{M}_{fg}^{\le n} ). 
  \]
\end{Rem}
\begin{Def}\label{def:franke}
	 Let $0 \le h \le n < \infty$, and let $BP_*/I_h \to R \ne 0$ be Landweber exact of height $n$ with $R$ even-periodic. Let $(R,\Gamma) \coloneqq (R, R\otimes_{BP_*}BP_*BP \otimes_{BP_*}R)$ be the associated flat Hopf algebroid, and set $L \coloneqq R_2$. The category of quasi-periodic $R_0R$-comodules $\dcat^{per}(\Comod_{R_0R})$ is
	\[
\dcat^{per}(\Comod_{R_0R}) \coloneqq \dcat^{(L,2)}(\Comod_{R_0R}). 
	\]
	For $k \ge 0$ we also define
	\[
\dcat^{per}(\Mod_{R_0R^{\otimes k}}) \coloneqq \dcat^{(R_0R^{\otimes k}\otimes L,2)}(\Mod_{R_0R^{\otimes k}})
	\]
	where in the case $k = 0$, we take $R_0R^{\otimes k} \cong R_0$. 
\end{Def}
\begin{Rem}
    Suppose $R$ and $F$ are even-periodic $p$-local Landweber exact homology theories of height $n$ as in the previous definition. Then, as noted previously, there is an equivalence of symmetric monoidal abelian categories
    \begin{equation}\label{eq:cat_equivalence}
\Comod_{R_*R} \simeq \Comod_{F_*F}.
    \end{equation}
    It follows that $\Comod_{R_0R} \simeq \Comod_{R_0R}$ as well, and hence 
\[
\dcat(\Comod_{R_0R}) \simeq \dcat(\Comod_{F_0F})
\]
as symmetric monoidal $\infty$-categories. We claim that we also have an equivalence $\dcat^{per}(\Comod_{R_0R}) \simeq \dcat^{per}(\Comod_{F_0F})$. This follows from \Cref{prop:br_module} if we can show that $P_{R_0R}(R_0)$ and $P_{F_0F}(F_0)$ correspond under the equivalence of categories. Let $L_R \coloneqq R_2$ and $L_F \coloneqq F_2$. Because the equivalence \eqref{eq:cat_equivalence} is monoidal the units $R_*$ and $F_*$ correspond (as graded rings), and hence $R_0$ and $F_0$ correspond (as well as $L_R$ and $L_F$). In other words, 
\[
P_{R_0R}(R_0) = \bigoplus_{k \in \mathbb{Z}} E_0 \otimes L_R^{\otimes k}
\]
corresponds to 
\[
P_{F_0F}(F_0) = \bigoplus_{k \in \mathbb{Z}} F_0 \otimes L_F^{\otimes k}
\]
and vice-versa, and the claim follows. 
\end{Rem}
\begin{Rem}\label{rem:algebraic_results}
    Taking $h = 0$, and letting $BP_* \to E \ne 0$ be Landweber exact of height $n$ with $E$ even-periodic, the category $\dcat^{per}(\Comod_{E_0E})$ is also known Franke's comodule category, as it was essentially introduced in the unpublished paper \cite{franke1996uniqueness} (Franke works with model categories and derivators instead of stable $\infty$-categories). The main result of \cite{BarthelSchlankStapleton2020Chromatic} is that for any non-principal ultrafilter $\mathcal{F}$ on the set of primes, there is an equivalence of symmetric monoidal stable $\infty$-categories
    \[
\prod_{\mathcal F}\Sp_n \simeq \prod_{\mathcal F} \dcat^{per}(\Comod_{E_0E}). 
    \]
    In a related result, Pstr{\c a}gowski \cite{piotr_paper} shows that for $2p-2 > 2(n^2+n)$ and $k = 2p-2-n^2-n$ there is an equivalence \[
h_k\Sp_n \simeq h_k\dcat^{per}(\Comod_{E_0E})
    \]
    between the homotopy $k$-categories of the $\infty$-category of $E$-local spectra and the category of quasi-periodic $E_0E$-comodules. More recently, in \cite[Theorem 8.13]{patchkoria2021adams} Patchkoria and Pstr{\c a}gowski  have shown that the bound can be improved to $2p-2 >n^2+n$. 
\end{Rem}
We note the following, which will be used later (we begin to omit subscripts from the periodization functors, as they behave as expected; for example $P_{E_0E}(E_0)$ and $P_{E_0}(E_0)$ agree as $E_0$-modules). 
\begin{Lem}
	\[
P(E_0) \cong E_* \quad \text{ and } \quad P(E_0E) \cong E_*E. 
	\]
\end{Lem}
\begin{proof}
	By definition $P(E_0) \cong \bigoplus_{k \in \mathbb{Z}} (E_0 \otimes L^{\otimes k}[-2k]) \cong E_*$ because $E$ is even-periodic. A similar argument work for $E_0E$. 
\end{proof}

\begin{Thm}\label{thm:descent_franke}
		Let $E$ be an even-periodic $p$-local Landweber exact (over $BP_*$) homology theory of height $n$, and let $L \coloneqq \pi_2(E)$. Suppose $p > n+1$, then $(E_0,E_0E)$ is a descendable Hopf algebroid, and hence there is an equivalence of symmetric monoidal stable $\infty$-categories
\[
\dcat^{per}(\Comod_{E_0E}) \simeq \Tot\left(\!\begin{tikzcd}
	{\dcat^{per}(\Mod_{E_0})} & {\dcat^{per}(\Mod_{E_0E})} & \cdots
	\arrow[shift left=1, from=1-2, to=1-3]
	\arrow[shift right=1, from=1-1, to=1-2]
	\arrow[shift left=1, from=1-1, to=1-2]
	\arrow[shift right=1, from=1-2, to=1-3]
	\arrow[from=1-2, to=1-3]
\end{tikzcd} \right)
\]

\end{Thm}
\begin{proof}
For the claim that $E_0E \in \cat D(\Comod_{E_0E})$ is descendable, see \cite[Lemma 5.30]{BarthelSchlankStapleton2020Chromatic} or \cite[Remark 4.14]{BarthelHeard2018Algebraic} (note that the property of being a descendable Hopf algebroid is preserved under the equivalence of categories of \Cref{thm:naumann}). The equivalence is then a consequence of \Cref{thm:descent_hopf} and the definitions (\Cref{def:franke}). 
\end{proof}
\begin{Rem}\label{Rem:endomorphism_spectrum}
	Suppose $\cat C$ is a compactly generated symmetric monoidal stable $\infty$-category generated by its tensor unit $\unit$. Let $E \in \CAlg(\cat C)$, then $\Mod_{\cat C}(E)$ is compactly generated by $E$. By adjunction, we see that the endomorphism spectrum $\Hom_{\Mod_{C}(E)}(E,E)$ then satisfies 
	\[
\pi_*\Hom_{\Mod_{\cat C}(E)}(E,E) \cong \pi_*\Hom_{\cat C}(\unit,E)
	\]
	Applying this with $\cat C = \dcat^{per}(\Mod_{E_0E^{\otimes k}})$ (which is compactly generated by its tensor unit and $E = P(E_0E^{\otimes k}) \cong E_*E^{\otimes k}$, we see that 
	\[
\pi_*\Hom_{\dcat^{per}(\Mod_{E_0E^{\otimes k}})}(P(E_0E^{\otimes k}),P(E_0E^{\otimes k})) \cong E_*E^{\otimes k}
	\]
\end{Rem}
\begin{Rem}\label{rem:bss_descent_compare}
	The description of $\dcat^{per}(\Comod_{E_0E})$ is given in a slightly different form in \cite{BarthelSchlankStapleton2020Chromatic}. We explain the connection here. 
	We will use the notation from \cite[Definition 4.2]{BarthelSchlankStapleton2020Chromatic}: for a spectrum $X$, we let 
	\[
X_{\star} \coloneqq H\pi_*(X),
	\]
	where $H \colon \GrAb \to \Sp$ is the functor from $\mathbb{Z}$-graded abelian groups to spectra, given by taking the generalized Eilenberg--Maclane spectrum. 

Using Schwede--Shipley Morita theory (\cite[Theorem 3.3.3]{SchwedeShipley2003Stable} and \cite[Theorem 7.1.2.1]{lurie-higher-algebra}), and arguing as in the previous remark (or applying \cite[Lemma 5.32]{BarthelSchlankStapleton2020Chromatic}) we have 
\[
\dcat^{per}(\Mod_{E_0E^{\otimes k}}) \cong \Mod_{(E^{\wedge {k+1}})_{\star}}. 
\]

	Then, we have the claimed result: there is a symmetric monoidal equivalence of stable $\infty$-categories
\[
\dcat^{per}(\Comod_{E_0E}) \simeq \Tot\left(\!\begin{tikzcd}
	\Mod_{E_{\star}} & \Mod_{(E \wedge E)_{\star}} & \cdots
	\arrow[shift left=1, from=1-2, to=1-3]
	\arrow[shift right=1, from=1-1, to=1-2]
	\arrow[shift left=1, from=1-1, to=1-2]
	\arrow[shift right=1, from=1-2, to=1-3]
	\arrow[from=1-2, to=1-3]
\end{tikzcd} \right)
\]
This is implicit, although not explicitly stated, in \cite[Section 5]{BarthelSchlankStapleton2020Chromatic}. 
\end{Rem}
\begin{Rem}
    One also has a similar result for the Hopf algebroid appearing in \Cref{exa:height_n}. More generally, we have the following:
\end{Rem}
\begin{Thm}\label{thm:descent_mrw}
        Let $K$ be an even-periodic $p$-local Landweber exact (over $BP_*/I_n$) homology theory of height $n$, let $(K,\Sigma)$ be the associated Hopf algebroid, and let $L \coloneqq \pi_2(K)$. Suppose $p-1 \nmid n$, then $(K_0,\Sigma_0)$ is a descendable Hopf algebroid, and hence there is an equivalence of symmetric monoidal stable $\infty$-categories
\[
\dcat^{per}(\Comod_{\Sigma_0}) \simeq \Tot\left(\!\begin{tikzcd}
    {\dcat^{per}(\Mod_{K_0})} & {\dcat^{per}(\Mod_{\Sigma_0})} & \cdots
    \arrow[shift left=1, from=1-2, to=1-3]
    \arrow[shift right=1, from=1-1, to=1-2]
    \arrow[shift left=1, from=1-1, to=1-2]
    \arrow[shift right=1, from=1-2, to=1-3]
    \arrow[from=1-2, to=1-3]
\end{tikzcd} \right)
\]
\end{Thm}
\begin{proof}
    The descendability of the Hopf algebroid $(K_0,\Sigma_0)$ is a consequence of \cite[Lemma 4.5]{BarthelHeard2018Algebraic}; now apply \Cref{thm:descent_hopf}. 
\end{proof}
\section{Picard groups}
In this section, we turn to our first application of descent. We recall that an object $M$ in a symmetric monoidal category is \emph{invertible} if there exists an object $M^{-1}$ such that $M \otimes M^{-1} \simeq \unit$. 
\begin{Def}
	For a symmetric monoidal presentable stable $\infty$-category $\cat C$, we let $\Pic(\cat C)$ denote the group of isomorphism classes of invertible objects in $\Ho(\cat C)$. 
\end{Def}
\begin{Rem}\label{rem:picard_groups_chromatic}
    The idea of studying Picard groups of localized categories of spectra was begun by Hopkins \cite{HopkinsMahowaldSadofsky1994Constructions}, who studied the Picard group of $K(n)$-locally invertible spectra. Hovey and Sadofsky then studied the Picard group of the category of $E(n)$-local spectra \cite{HoveySadofsky1999Invertible}. Their Theorem A shows that for $2p-2 > n^2+n$, we have $\Pic(\Sp_{E(n)}) \cong \mathbb{Z}$, generated by $ L_n S^1$. Given \Cref{rem:algebraic_results}, our main result in this section, \Cref{thm:hovey_strickland_local}, is a direct analog of this theorem. 
\end{Rem}
\begin{Rem}\label{rem:picard_spectrum_properties}
	We recall from \cite{MathewStojanoska2016Picard} that to a symmetric monoidal presentable stable $\infty$-category $\cat C$ we can associate a connective spectrum $\pics(\cat C)$ whose homotopy groups are given by the following:
	\begin{equation}\label{pic:defn}
	\pi_t\pics(\cat C) \cong
	\begin{cases}
		\Pic(\cat C) & t = 0, \\
		\pi_0\End_{\cat C}(\unit_{\cat C},\unit_{\cat C})^{\times} & t =1 \\
		\pi_{t-1}\End_{\cat C}(\unit_{\cat C},\unit_{\cat C}) & t \ge 2. 
	\end{cases}
	\end{equation}
	Moreover, as a functor $\Cat^{\otimes} \to \Sp_{\ge 0}$ from the $\infty$-category of symmetric monoidal stable $\infty$-categories to the $\infty$-category $\Sp_{\ge 0}$ of connective spectra, $\pics$ commutes with limits \cite[Proposition 2.2.3]{MathewStojanoska2016Picard}. 
\end{Rem}
    \begin{Prop}\label{prop:totalization_pic_descendable}
    Suppose $(A,\Gamma)$ is a descendable Hopf algebroid, and $L$ is an invertible $\Gamma$-comodule, then there is an equivalence of connective spectra
\[
\resizebox{\columnwidth}{!}{$\displaystyle
\pics(\dcat^{(L,N)}(\Comod_{\Gamma}))\simeq \tau_{\ge 0}\Tot\left(
\begin{tikzcd}[ ampersand replacement=\&, column sep=1em]
{\pics(\dcat^{(L,N)}(\Mod_{A}))}\&{\pics(\dcat^{(\Gamma \otimes L,N)}(\Mod_{\Gamma}))} \& \cdots 
    \arrow[shift left=1, from=1-2, to=1-3]
    \arrow[shift right=1, from=1-1, to=1-2]
    \arrow[shift left=1, from=1-1, to=1-2]
    \arrow[shift right=1, from=1-2, to=1-3]
    \arrow[from=1-2, to=1-3]
\end{tikzcd} \right)
$}
\]
\end{Prop}
\begin{proof}
    This follows from \Cref{thm:descent_hopf} and the fact that $\pics$ commutes with limits. 
\end{proof}
Applying this to \Cref{thm:descent_franke} we get the following:
	\begin{Prop}\label{prop:totalization_pic}
Let $E$ be an even-periodic $p$-local Landweber exact (over $BP_*$) homology theory of height $n$, and let $L \coloneqq \pi_2(E)$, then for $p > n+1$ there is an equivalence of connective spectra
\[
\resizebox{\columnwidth}{!}{$\displaystyle
\pics(\dcat^{per}(\ComodE))\simeq \tau_{\ge 0}\Tot\left(
\begin{tikzcd}[ ampersand replacement=\&, column sep=1em]
{\pics(\dcat^{per}(\Mod_{E_0}))}\&{\pics(\dcat^{per}(\Mod_{E_0E}))} \& \cdots 
    \arrow[shift left=1, from=1-2, to=1-3]
    \arrow[shift right=1, from=1-1, to=1-2]
    \arrow[shift left=1, from=1-1, to=1-2]
    \arrow[shift right=1, from=1-2, to=1-3]
    \arrow[from=1-2, to=1-3]
\end{tikzcd} \right)
$}
\]
\end{Prop}
\begin{Rem}\label{rem:invertible_cohomology}
  Using descent (in particular, combining \cite[Proposition 3.19]{Mathew2016Galois} and \cite[Lemma 6.1]{heard_kn}) one deduces that $X \in \dcat^{per}(\ComodE)$ is invertible if and only if $P_{\Gamma}(E_0) \otimes X$ is invertible in $\Mod_{\dcat^{per}(\ComodE)}(P_{\Gamma}(E_0))$, or equivalently (via \Cref{prop:module_categories_descendable}) $\rho_*X \in \dcat^{per}(\Mod_{E_0})$ is invertible.  We shall see in a moment that $\Pic(\dcat^{per}(\Mod_{E_0})) \cong \mathbb{Z}/2$ generated by $P(E_0)[1] \cong E_*[1]$. Together, 
  \[
X \in \Pic(\dcat^{per}(\ComodE)) \iff H_*(\rho_*(X)) \cong E_* \text{ (up to a shift)}.
  \]
  This is the analog of the result, implicit in Hovey--Sadofsky \cite{HoveySadofsky1999Invertible}
  \[
X \in \Pic(\Sp_n) \iff \pi_*(E \wedge X) \cong E_* \text{ (up to a shift)}. 
  \]
\end{Rem}
We first study the spectral sequence associated to $\pics(\dcat^{per}(\Comod_{E_0E}))$. 
\begin{Thm}\label{cor:en_ss}
With $E$ as in the previous proposition, suppose $p > n+1$, then there is a spectral sequence
	\[
E_2^{s,t} \cong \begin{cases}
	\mathbb{Z}/2 & s = t = 0 \\
	H^s(\mathcal{M}_{fg}^{\le n},\mathcal{O}_{\mathcal{M}_{fg}}^{\times}) & t = 1\\
	\Ext_{E_*E}^{s,t-1}(E_*,E_*) & t \ge 2. 
\end{cases}
	\]
	which converges for $t -s \ge 0$ to $\pi_{t-s}\pics(\dcat^{per}(\Comod_{E_0E}))$. The differentials run $d_r \colon E_r^{s,t} \to E_r^{s+r,t+r-1}$. 
\end{Thm}
\begin{proof}
This is the Bousfield--Kan spectral sequence associated to the totalization in \Cref{prop:totalization_pic}. Given the descent result above, the proof is in fact completely analogous to \cite[Theorem 3.4.3]{MathewStojanoska2016Picard}. 

Suppose first that $t \ge 2$. Then by \Cref{Rem:endomorphism_spectrum} we have 
\[
\begin{split}
\pi_t\pics(\dcat^{per}(\Mod_{E_0E^{\otimes k}})) &\cong \pi_{t-1} \Hom_{\dcat^{per}(\Mod_{E_0E^{\otimes k}})}(P(E_0E^{\otimes k}),P(E_0E^{\otimes k}))\\
& \cong (E_*E^{\otimes k})_{t-1}.
\end{split}
\]
 Then, the $E_2^{s,t}$-term of the spectral sequence is the $s$-th cohomology of the complex
\[\begin{tikzcd}
	{(E_*)_{t-1}} & {(E_*E)_{t-1}} & {(E_*E^{\otimes 2})_{t-1}} & \cdots
	\arrow[shift left=1, from=1-2, to=1-3]
	\arrow[shift left=1, from=1-3, to=1-4]
	\arrow[shift right=1, from=1-3, to=1-4]
	\arrow[shift left=3, from=1-3, to=1-4]
	\arrow[shift right=1, from=1-1, to=1-2]
	\arrow[shift left=1, from=1-1, to=1-2]
	\arrow[shift right=1, from=1-2, to=1-3]
	\arrow[from=1-2, to=1-3]
	\arrow[shift right=3, from=1-3, to=1-4]
\end{tikzcd}\]
Unwinding the definition of the maps, we see that this is precisely the cohomology of the cobar complex, and so this is isomorphic to $\Ext^{s,t-1}_{E_*E}(E_*E,E_*E)$. 

In the case $t = 1$, this is the cohomology of the units of the complex
\[\begin{tikzcd}
	{E_0} & {E_0E} & {E_0E^{\otimes 2}} & \cdots
	\arrow[shift left=1, from=1-2, to=1-3]
	\arrow[shift left=1, from=1-3, to=1-4]
	\arrow[shift right=1, from=1-3, to=1-4]
	\arrow[shift left=3, from=1-3, to=1-4]
	\arrow[shift right=1, from=1-1, to=1-2]
	\arrow[shift left=1, from=1-1, to=1-2]
	\arrow[shift right=1, from=1-2, to=1-3]
	\arrow[from=1-2, to=1-3]
	\arrow[shift right=3, from=1-3, to=1-4]
\end{tikzcd}\]
Passing to the algebro-geometric approach, we note that the simplicial scheme 
\[\begin{tikzcd}
	\cdots & { \Spec E_0E} & {\Spec E_0}
	\arrow[shift left=1, from=1-1, to=1-2]
	\arrow[shift right=1, from=1-1, to=1-2]
	\arrow[from=1-1, to=1-2]
	\arrow[shift left=1, from=1-2, to=1-3]
	\arrow[shift right=1, from=1-2, to=1-3]
\end{tikzcd}\]
is a presentation for $\mathcal{M}_{fg}^{\le n}$. In particular, the spectral sequence has the claimed form when $t =1 $. 

In the case $s = t = 0$, this is precisely the Picard group $\Pic(\dcat^{per}(\Mod_{E_0}))$. To see that this is $\mathbb{Z}/2$, one can use that  $\dcat^{per}(\Mod_{E_0}) \simeq \Mod_{E_{\star}}$ and appeal to \cite[Theorem 8.1]{BakerRichter2005Invertible}. See also \cite[Lemma 5.33]{BarthelSchlankStapleton2020Chromatic}.
\end{proof}
\begin{Rem}
	Note that (as expected) this spectral sequence has the same $E_2$-term as the corresponding spectral sequence for the Picard spectrum of the $E(n)$-local category, see \cite[Theorem 3.4.3]{MathewStojanoska2016Picard}. Here they use the more invariant description of the $E_2$-term as 
  \[
E_2^{s,t} = H^s(\mathcal{M}_{fg}^{\le n},\omega^{(t-1)/2})
  \]
  for $t \ge 3$. 
\end{Rem}
Putting this all together, we can compute the Picard group of Franke's comodule category at large primes. 
\begin{Thm}\label{thm:hovey_strickland_local}
Let $E$ be an even-periodic $p$-local Landweber exact (over $BP_*$) homology theory of height $n$, and let $L \coloneqq \pi_2(E)$. Suppose $2p-2 > n^2+n$, then $\Pic(\dcat^{per}(\Comod_{E_0E})) \cong \mathbb{Z}$, generated by $ P(E_0)[1] \cong E_*[1]$. 
\end{Thm}
\begin{proof}
This is the claim that $\pi_0(\pics(\dcat^{per}(\Comod_{E_0E}))) \cong \mathbb{Z}$. We will prove this via the spectral sequence of \Cref{cor:en_ss} (which applies because $2p-2 > n^2+n$ implies $p> n+1$). 

By \cite[Theorem 5.1]{HoveySadofsky1999Invertible} we have 
	\[
\Ext_{E_*E}^{s,\ast}(E_*,E_*) =0 
	\]
	for $s > n^2+n$. In fact, Hovey--Sadofsky use Johnson--Wilson $E$-theory, but this does not change anything in light of \cite[Theorem C]{HoveyStrickl2005Comodules} (or \Cref{thm:naumann}). Now, a sparseness argument shows that in the stable range the spectral sequence is zero unless the internal degree is a multiple of $2(p-1)$. It follows that there is no room for non-trivial differentials or extensions in the stable range (as non-zero differentials raise filtration by a multiple of $2p-1$). 

	By \cite[Proposition 3.4.2]{MathewStojanoska2016Picard} we have $	H^1(\mathcal{M}_{fg}^{\le n},\mathcal{O}_{\mathcal{M}_{fg}}^{\times}) \cong \mathbb{Z}$, generated by the tautological line bundle $\omega$. This corresponds to $E_{\ast +2 } \simeq E_*[2]$, and hence to the periodic comodule $P(E_0)[2]$. It follows that this class survives the spectral sequence.  The copy of $\mathbb{Z}/2$ in degree $(0,0)$ corresponds to $P(E_0)[1] \cong E_*[1]$, the generator of the Picard group $\Pic(\dcat^{per}(\Mod_{E_0}))$. We therefore have a non-trivial extension
	\[
0 \to \mathbb{Z} \to \Pic(\dcat^{per}(\Comod_{E_0E})) \to \mathbb{Z}/2 \to 0,
	\]
	where the first map is multiplication by $2$. It follows that $\Pic(\dcat^{{per}}(\Comod_{E_0E})) \cong \mathbb{Z}$, generated by $P(E_0)[1] \cong E_*[1]$, as claimed. 
\end{proof}
\begin{Rem}[The Galois group]
    The Picard group is not the only invariant of $\dcat^{per}(\ComodE)$ that can be studied via descent. For example, in \cite{Mathew2016Galois} Mathew introduces an invariant of a symmetric monoidal stable $\infty$-category $\mathcal{C}$, the Galois group(oid) $\pi_{1}\cat C$. For example, if $A$ is a commutative ring spectrum, then a continuous group homomorphism $\pi_1\Mod_A \to G$ is equivalent to giving a faithful $G$-Galois extension of $A$ \cite{Rognes2008Galois}. If $R$ denotes the endomorphism ring of the unit object in $\cat C$, then there is always a canonical surjection
\[
\pi_1(\cat C) \to \pi_1^{et}\Spec(R)
\]
to the \'etale fundamental group of $\Spec(R)$. One says that the Galois group of $\cat C$ is algebraic if this canonical surjection is an isomorphism. 

    By \cite[Theorem 10.15]{Mathew2016Galois} the Galois theory of the $E(n)$-local category $\Sp_n$ is algebraic; the Galois group $\pi_1(\Sp_n)$ is isomorphic to the \'{e}tale fundamental group of $\Spec\mathbb{Z}_{(p)}$. One therefore expects that when $p > n+1$ the Galois theory of $\dcat^{per}(\ComodE)$ is algebraic, and indeed, this is the case. In fact, using descent, this is essentially the same argument as in \cite[Theorem 10.15]{Mathew2016Galois}. We leave the details to the interested reader.
\end{Rem}
\begin{Rem}
  One can also study the Hopf algebroids $(K,\Sigma)$ appearing in \Cref{thm:descent_mrw}. We will use the notation of \Cref{exa:height_n} so, $\mathbb{G}_n$ denotes the Morava stabilizer group. We note that $\Ext_{\Sigma}^{*,*}(K_*,K_*)$ is a $K(n)_*$-module (\cite[Proposition 5.1.12]{Ravenel1986Complex}), and so is $2(p^n-1)$-periodic. This translates into the category $\dcat^{(L,2)}(\Comod_{\Sigma_0})$ having a $2(p^n-1)$-periodicity, i.e., $P_{\Sigma}(K_0)[2(p^n-1)] \cong P_{\Sigma}(K_0) \cong K_*$. 
\end{Rem}
\begin{Thm}\label{thm:pic_mrw}
    Let $K$ be an even-periodic $p$-local Landweber exact (over $BP_*/I_n$) homology theory of height $n$, let $(K,\Sigma)$ be the associated Hopf algebroid. Suppose $2p-2 \ge n^2$ and $p-1 \nmid n$, then $\Pic(\dcat^{per}(\Comod_{\Sigma_0})) \cong \mathbb{Z}/(2(p^n-1))$ generated by $K_*[1] \cong P_{\Sigma}(K_0)[1]$. 
\end{Thm}
\begin{proof}
   The argument is similar to that used in \Cref{thm:hovey_strickland_local}. First we note that since we are free to make a choice of the height $n$ homology theory used, we take $K_n \coloneqq E_n/I_n$, so that we use $(K_{n,\ast},\Sigma_{n,\ast})$ from \Cref{exa:height_n}. Using descent, we claim that we have a spectral sequence analogous to that in \Cref{cor:en_ss}, namely 
        \[
E_2^{s,t} \cong \begin{cases}
    \mathbb{Z}/2 & s = t = 0 \\
    H^s(\mathbb{G}_n,\mathbb{F}_{p^n}^{\times}) & t = 1\\
    \Ext_{\Sigma_n}^{s,t-1}(K_{n,\ast},K_{n,\ast}) & t \ge 2. 
\end{cases}
    \]
    which converges for $t -s \ge 0$ to $\pi_{t-s}\pics(\dcat^{per}(\Comod_{\Sigma_0})) $. To identify the $E_2^{0,0}$-term we observe that this is exactly $\Pic(\dcat^{per}(\Mod_{K_0}) \cong \Pic(\Mod_{K_{\star}}) \cong \bbZ/2$ by \cite[Theorem 8.1]{BakerRichter2005Invertible}. Perhaps the only further point of note here is the identification of the $t = 1$ line. 
This is the cohomology of the complex
\[\begin{tikzcd}
  {K_{n,0}^{\times}} & {\Sigma_{n,0}^{\times}} & ({\Sigma_{n,0}}\otimes_{K_{n,0}} \Sigma_{n,0})^{\times} & \cdots
  \arrow[shift left=1, from=1-2, to=1-3]
  \arrow[shift left=1, from=1-3, to=1-4]
  \arrow[shift right=1, from=1-3, to=1-4]
  \arrow[shift left=3, from=1-3, to=1-4]
  \arrow[shift right=1, from=1-1, to=1-2]
  \arrow[shift left=1, from=1-1, to=1-2]
  \arrow[shift right=1, from=1-2, to=1-3]
  \arrow[from=1-2, to=1-3]
  \arrow[shift right=3, from=1-3, to=1-4]
\end{tikzcd}\]
We now note that $\Sigma_{n,0} \cong \Hom^c(\mathbb{G}_n,K_{n,0})$, see \cite{Hovey2004Operations} or \cite[Theorem 12]{Strickl2000GrossHopkins} (reduced modulo $I_n$). Then $\Sigma_{n,0}^{\otimes k} \cong \Hom^c(\mathbb{G}_n^{\times k},K_{n,0})$; this argument is essentially contained in Appendix II of \cite{Devinatz1995Moravas}. Noting that $K_{n,0} \cong \mathbb{F}_{p^n}$, we see that we are computing the cohomology of the complex 
\[\begin{tikzcd}
  {\mathbb{F}_{p^n}^{\times}} & {\Hom^c(\mathbb{G}_n,\mathbb{F}_{p^n}^{\times})} & {\Hom^c(\mathbb{G}_n\times\mathbb{G}_n,\mathbb{F}_{p^n}^{\times})} & \cdots
  \arrow[from=1-1, to=1-2]
  \arrow[shift left=1, from=1-2, to=1-3]
  \arrow[shift right=1, from=1-2, to=1-3]
  \arrow[from=1-3, to=1-4]
  \arrow[shift left=1, from=1-3, to=1-4]
  \arrow[shift right=1, from=1-3, to=1-4]
\end{tikzcd}\]
which is exactly $H^s(\mathbb{G}_n,\mathbb{F}_{p^n}^{\times})$.

    We have $E_2^{s,t}  = 0$ for $s>n^2$ and the usual sparseness argument shows that there are no non-trivial differentials in the stable range. Then, the spectral sequence only has two non-zero terms in the $t - s =0$ column, namely 
    \[
E_2^{0,0} \cong \bbZ/2 \quad  \text{ and } \quad E_2^{1,1} \cong H^{1}(\mathbb{G}_n,\mathbb{F}_{p^n}^{\times}).
    \] 
    We show below in \Cref{prop:group_coh_calculations} that $H^{1}(\mathbb{G}_n,\mathbb{F}_{p^n}^{\times}) \cong \bbZ/(p^n-1)$, generated by a certain class $\eta$. This class comes from the degree two class in $K_n$, and corresponds to $K_*[2]$. Once again, we deduce that there is a non-trivial extension
    \[
0 \to \bbZ/(p^n-1) \to      
\Pic(\dcat^{per}(\Comod_{\Sigma_0})) \to \bbZ/2 \to 0,
    \]
    and the result follows. 
\end{proof}
\begin{Def}
    Let $t_0 \colon \mathbb{G}_n \to (E_n)_0^{\times}$ denote the crossed homomorphism given by
    \[
    \begin{split}
t_0 \colon \mathbb{G}_n &\to (E_n)_0^{\times}\\
g & \mapsto \frac{g_*u}{u}.
\end{split}
    \]
    and then let $\eta$ denote the composite $\eta \colon \mathbb{G}_n \xrightarrow{t_0} (E_n)_0^{\times} \twoheadrightarrow \mathbb{F}_{p^n}^{\times}$.  
\end{Def}
\begin{Prop}\label{prop:group_coh_calculations}
    Suppose $2p-2 \ge n^2$ and $p-1 \nmid n$, then we have $H^1(\mathbb{G}_n,\mathbb{F}_{p^n}^{\times}) \cong \mathbb{Z}/(p^n-1)$ generated by $\eta$. 
\end{Prop}
\begin{proof}
    The details are essentially contained in the thesis of Lader \cite{lader2013resolution}. Since this is unpublished (and in French) we spell out some of the details here. We need to introduce some notation, compare \cite[Section 1.1 and 1.3]{ghmr}.

    Let $S_n$ denote the $p$-Sylow subgroup of $\mathbb{S}_n$. We recall that there is a determinant homomorphism $\det \colon \mathbb{G}_n \to \mathbb{Z}_p^{\times}$. Identifying the quotient of $\mathbb{Z}_p^{\times}$ by its maximal finite subgroup $\mu_{p^n-1}$ with $\mathbb{Z}_p$, we get the reduced determinant map
    \[
N  \colon \mathbb{G}_n \to \mathbb{Z}_p.
    \]
    Let $\mathbb{G}_n^1$ be the kernel of this map, and $S_n^1$ denote its restriction to $S_n$. Note that $\mathbb{G}_n \simeq \mathbb{G}_n^1 \times \mathbb{Z}_p$ (in general this is a semi-direct product, but our assumptions imply that $p \nmid n$, so that this is actually a product) and $\mathbb{G}_n^1 \cong S_n^1 \rtimes F$, for $F \cong \mu_{p^n-1} \times \Gal(\mathbb{F}_{p^n}/\mathbb{F}_p)$ \cite[Lemma 1.13]{lader2013resolution}. 

    Because $\bbZ_p$ has cohomological dimension 1, the (collapsing) Lyndon--Hochschild--Serre spectral sequence gives a short exact sequence 
    \[
0 \to H^1(\mathbb{Z}_p,H^0(\mathbb{G}_n^1,\bbF_{p^n}^{\times})) \to H^1(\mathbb{G}_n,\mathbb{F}_{p^n}^{\times}) \to H^0(\bbZ_p,H^1(\mathbb{G}_n^1,\mathbb{F}_{p^n}^{\times})) \to 0. 
    \]
We have that $H^1(\mathbb{Z}_p,H^0(\mathbb{G}_n^1,\bbF_{p^n}^{\times})) = 0$, because $\mathbb{Z}_p$ is a pro-$p$ group, and 
\[
H^0(\bbZ_p,H^1(\mathbb{G}_n^1,\mathbb{F}_{p^n}^{\times})) \cong H^1(\mathbb{G}_n^1,\mathbb{F}_{p^n}^{\times})
\]
as $\bbZ_p$ is acting trivially. Therefore, we deduce that
\[
H^1(\mathbb{G}_n,\mathbb{F}_{p^n}^{\times}) \cong H^1(\mathbb{G}_n^1,\mathbb{F}_{p^n}^{\times}).
\]

    We have a spectral sequence
    \[
H^r(F,H^s(S_n^1,\mathbb{F}_{p^n}^{\times})) \implies H^{r+s}(\mathbb{G}_n^1,\mathbb{F}_{p^n}^{\times}).
    \]
    Because $S_n^1$ is a pro-$p$ group and $\mathbb{F}_{p^n}^{\times}$ has order prime to $p$, we have $H^s(S_n^1,\mathbb{F}_{p^n}^{\times}) = 0$ for $s > 0$ (\cite[Proposition A.22]{lader2013resolution}). Therefore, we have 
    \[
H^s(S_n^1,\mathbb{F}_{p^n}^{\times}) \cong \begin{cases}
    0 & q > 0 \\
    \mathbb{F}_{p^n}^{\times} & s = 0,
\end{cases}
    \]
    and the spectral sequence collapses to give $H^r(\mathbb{G}_n^1,\mathbb{F}_{p^n}^{\times}) \cong H^r(F,\mathbb{F}_{p^n}^{\times})$. Finally, when $r = 1$ this is equal to $\mathbb{Z}/(p^n-1)$ by \cite[Proposition 5.19]{lader2013resolution} (this uses the relevant Lyndon--Hochschild--Serre sequence); more specifically, $H^1(F;\mathbb{F}_{p^n}^{\times}) \cong H^1(\mu_{p^n-1};\mathbb{F}_{p^n}^{\times}) \cong \End(\mathbb{F}_{p^n}^{\times}) \cong \mathbb{Z}/(p^n-1)$.

    Finally, we show that the generator corresponds to the class $\eta$ (this is also \cite[Corollary 5.24(3)]{lader2013resolution}, where $\eta$ is denoted $\Omega$). Unwinding the definitions, the isomorphism $H^1(\mathbb{G}_n,\bbF_{p^n}^{\times}) \cong \bbZ/(p^n-1)$ is given by 
\[\begin{tikzcd}[row sep=0.1em, column sep = 10em]
  {H^1(\mathbb{G}_n,\mathbb{F}_{p^n}^{\times})} & {\End(\mathbb{F}_{p^n}^{\times})} \\
  {[\mathbb{G}_n \xrightarrow{d}\mathbb{F}_{p^n}^{\times}]} & {[\mathbb{F}_{p^n}^{\times} \hookrightarrow \mathbb{G}_n \xrightarrow{d}\mathbb{F}_{p^n}^{\times}]}
  \arrow[from=1-1, to=1-2]
  \arrow[maps to, from=2-1, to=2-2]
\end{tikzcd}\]
Then, the composite
\[
\mathbb{F}_{p^n}^{\times} \hookrightarrow \mathbb{G}_n \xrightarrow{t_0} (E_n)_0^{\times} \twoheadrightarrow\mathbb{F}_{p^n}^{\times}
\]
is the identity (use \cite[Equation (1.5) and Proposition 1.4]{lader2013resolution}), and the result follows. 
\end{proof}

\section{\texorpdfstring{$I_n$}{I}-complete quasi-periodic comodules}\label{def:complete_comodules}
In \cite{BarthelSchlankStapleton2021Monochromatic} Barthel--Schlank--Stapleton extended their algebraicity results to the $K(n)$-local setting. The analog of the comodule categories $\dcat^{per}(\ComodE)$ is given by the Bousfield localization at $P(E_0/I_n)$, analogous to how the $K(n)$-local category can be obtained from the $E_n$-local category by Bousfield localization at the localization of a finite local type $n$ spectrum. The goal of this section is to prove an analog of our descent results for this localized category, and to compute its Picard group. 
\begin{Conv}
  In this section we fix a choice of $p$-local Landweber exact cohomology theory of type $n$. Namely, we let $E = E_n$ denote Morava $E$-theory associated to the Honda formal group law over $\bbF_{p^n}$; this is because we will give results that depend on the cohomology of the associated Morava stabilizer group $\mathbb{G}_n$, and because the ring $E_*$ is already $I_n$-adically complete. 
\end{Conv}
\begin{Rem}
  We recall that in a symmetric monoidal stable $\infty$-category $(\cat C,\otimes,\unit)$ we say that an object $X \in \cat C$ is $A$-local (or $A$-complete) if for any $Y \in \cat C$ with $Y \otimes A \simeq 0$, the space of maps $\Hom_{\cat C}(Y,X)$ is contractible. 

  The inclusion of $A$-local objects in $\cat C$ into $\cat C$ has a left adjoint, and we let $L_A \colon \cat C \to L_A\cat C$ denote the corresponding localization functor. This functor is symmetric monoidal, and gives $L_A\cat C$ the structure of a symmetric monoidal stable $\infty$-category with the localized tensor product $X \htimes Y \coloneqq L_A(X \otimes Y)$. 
\end{Rem}
\begin{Def}\label{def:complete_categories}
We let
\[
\begin{split}
	\cat D(Mod_{E_0})^{\wedge}_{I_n} & \coloneqq L_{E_0/I_n}\dcat (\Mod_{E_0})\\
	\cat D(\ComodE)^{\wedge}_{I_n} & \coloneqq L_{E_0/I_n}\dcat (\ComodE) \\
	\dcat^{per}(Mod_{E_0})^{\wedge}_{I_n} & \coloneqq L_{P(E_0/I_n)}\dcat^{per} (\Mod_{E_0})\\
	\dcat^{per}(\ComodE)^{\wedge}_{I_n} & \coloneqq L_{P(E_0/I_n)}\dcat^{per} (\ComodE) \\
\end{split}
\]
We will usually simply write $(-)^{\wedge}_{I_n}$ for the localization (completion) functor. 
\end{Def}
\begin{Rem}
The category $\dcat^{per}(\ComodE)^{\wedge}_{I_n}$ seems to have first been considered by Barthel--Schlank--Stapleton \cite{BarthelSchlankStapleton2021Monochromatic}, who proved that for any non-principal ultrafilter $\mathcal{F}$ on the set of primes, there is an equivalence of symmetric monoidal stable $\infty$-categories
      \[
\prod_{\mathcal F}\Sp_{K(n)} \simeq \prod_{\mathcal F} \dcat^{per}(\Comod_{E_0E})^{\wedge}_{I_n}. 
    \]
    On the other hand, we are not aware of any similar result for the categories $\dcat^{per}(\Comod_{\Sigma_0})$ appearing in \Cref{thm:descent_mrw}. 
\end{Rem}
\begin{Rem}
We note that \Cref{prop:bs_complete} implies that the following diagram commutes:
\[\begin{tikzcd}
  {\dcat(\Comod_{E_0E})} & {\dcat^{per}(\Comod_{E_0E})} \\
  {\dcat(\Comod_{E_0E})^{\wedge}_{I_n}} & {\dcat^{per}(\Comod_{E_0E})^{\wedge}_{I_n}}
  \arrow["U"', from=1-2, to=1-1]
  \arrow["{(-)^{\wedge}_{I_n}}"', from=1-1, to=2-1]
  \arrow["U", from=2-2, to=2-1]
  \arrow["{(-)^{\wedge}_{I_n}}", from=1-2, to=2-2]
\end{tikzcd}\]

	In particular, $(UP(M))^{\wedge}_{I_n} \simeq U(P(M)^{\wedge}_{I_n})$ for any $M \in \dcat(\Comod_{E_0})$, where the completions are taken in the appropriate categories. Therefore, there is no ambiguity if we simply write $P(M)^{\wedge}_{I_n}$. The same holds true for $\Mod_{E_0}$ (or indeed, any ring $R$).
\end{Rem}
\begin{Rem}\label{rem:complete_free}
  Let $R$ be a commutative noetherian ring, $I$ a prime ideal, and let $(-)^{\wedge}_I \colon \dcat(\Mod_R) \to L_{R/I}\dcat(\Mod_R)$ denote the Bousfield localization at $R/I$. Then, for any $M \in \dcat(\Mod_R)$ there is a spectral sequence 
  \[
H_p^I(H_q(M)) \implies H_{p+q}(M^{\wedge}_{I})
  \]
  see for example, the last section of \cite{DwyerGreenlees2002Complete} (specifically, page 219). Here $H_p^I$ denotes the local homology groups of Greenlees--May \cite{GreenleesMay1992Derived}. Modules in the essential image of the functor $H_0^I$ are called $L$-complete, see \cite[Appendix A]{HoveyStrickl1999Morava}. We note that if $H_*M$ is a free, or even just flat, $R$-module, then we have $H_p^I(H_*M) = 0$ for $p>0$ and $H_0^I(H_*M) \cong (H_*M)^{\wedge}_I$ (for example, by \cite[Theorem A.2]{HoveyStrickl1999Morava}), and the spectral sequence collapses to give isomorphisms
  \[
H_0^I(H_*M) \cong (H_*M)^{\wedge}_{I} \cong H_*(M^{\wedge}_I).
  \]
  This is why we call this the completion functor. 
\end{Rem}
\begin{Lem}\label{lem:prop_internal_complete}
There are equivalences
	\[
	\begin{split}
\dcat^{per}(\Mod_{E_0})^{\wedge}_{I_n} \simeq \Mod_{\dcat(\Mod_{E_0})^{\wedge}_{I_n}}(P(E_0)^{\wedge}_{I_n}) \\
\dcat^{per}(\ComodE)^{\wedge}_{I_n} \simeq \Mod_{\dcat(\ComodE)^{\wedge}_{I_n}}(P(E_0)^{\wedge}_{I_n}) 
\end{split}
	\]
\end{Lem}
\begin{proof}
	This follows from \Cref{prop:bs_complete} with $T = E_0/I_n$ applied to the $(P,U)$ adjunction of \Cref{rem:pu_adjunction} .
\end{proof}
\begin{Prop}\label{prop:complete_module}
	 There is an equivalence 
	\[
\dcat^{per}(\Mod_{E_0})^{\wedge}_{I_n} \simeq \Mod_{\dcat^{per}(\ComodE)^{\wedge}_{I_n}}(P(E_0E)^{\wedge}_{I_n}). 
	\]
\end{Prop}
\begin{proof}
By \Cref{prop:bs_complete} (with $T = P(E_0/I_n)$) the adjunction 
\[
\rho_* \colon \dcat^{per}(\ComodE) \longleftrightarrows \dcat^{per}(\Mod_{E_0}) \colon \rho^*
\]
gives rise to an adjunction
\[
\epsilon_* \colon \dcat^{per}(\ComodE)^{\wedge}_{I_n} \longleftrightarrows \dcat^{per}(\Mod_{E_0})^{\wedge}_{I_n} \colon \epsilon^*
\]
and a commutative diagram
\begin{equation}\label{eq:commutative}\begin{tikzcd}
  \dcat^{per}(\ComodE) & \dcat^{per}(\Mod_{E_0}) \\
  \dcat^{per}(\ComodE)^{\wedge}_{I_n} & \dcat^{per}(\Mod_{E_0})^{\wedge}_{I_n}.
  \arrow["\rho^*"', from=1-2, to=1-1]
  \arrow["{(-)^{\wedge}_{I_n}}", shift right=1, from=1-2, to=2-2]
  \arrow["{\epsilon^*}", from=2-2, to=2-1]
  \arrow["{(-)^{\wedge}_{I_n}}"', shift right=1, from=1-1, to=2-1]
\end{tikzcd}\end{equation}

Moreover, there is an equivalence of symmetric monoidal $\infty$-categories 
\[
\dcat^{per}(\Mod_{E_0})^{\wedge}_{I_n}  \simeq \Mod_{\dcat^{per}(\ComodE)^{\wedge}_{I_n}}(\epsilon^*(P(E_0)^{\wedge}_{I_n})).
\]
To finish the proof, we observe that
\[
\epsilon^*(P(E_0)^{\wedge}_{I_n}) \simeq (\rho^*(P(E_0)))^{\wedge}_{I_n} \simeq P(E_0E)^{\wedge}_{I_n}. \qedhere
\]
\end{proof}
\begin{Lem}\label{lem:local_pic}
  The localization functor $(-)^{\wedge}_{I_n} \colon \dcat^{per}(\Mod_{E_0}) \to \dcat^{per}(\Mod_{E_0})^{\wedge}_{I_n}$ induces an isomorphism on Picard groups. In particular, 
  \[
\Pic( \dcat^{per}(\Mod_{E_0})^{\wedge}_{I_n}) \cong \bbZ/2. 
  \]
\end{Lem}
\begin{proof}
In light of \Cref{prop:free_forget_descendable,prop:complete_module}, this follows from \cite[Corollary 3.15]{BarthelSchlankStapleton2021Monochromatic} (to translate - their category $\Fr_{n,p}$ is our $\dcat^{per}(\ComodE)$, $A_{n,p}$ corresponds to $P(E_0E)$ and $\widehat \Mod_{A_{n,p}}$ corresponds to $\Mod_{\dcat^{per}(\ComodE)^{\wedge}_{I_n}}(P(E_0E)^{\wedge}_{I_n})$). 
\end{proof}
The following is the completed version of \Cref{prop:module_categories_descendable}. 

\begin{Prop}\label{prop:descendable_module_complete}
  For each $k \ge 2$ there is an equivalence of symmetric monoidal stable $\infty$-categories 
  \[
  \begin{split}
\Mod_{\dcat^{per}(\ComodE)^{\wedge}_{I_n}}((P(E_0E)^{\wedge}_{I_n})^{\htimes k}) &\simeq \Mod_{\dcat^{per}(\Mod_{E_0})^{\wedge}_{I_n}}((P(E_0E)^{\wedge}_{I_n})^{\htimes (k-1)}). 
\end{split}
  \]
\end{Prop}
\begin{proof}
This is essentially the same proof as in \Cref{prop:module_categories_descendable}. We begin with the adjunction
\[
\epsilon_* \colon \dcat^{per}(\ComodE)^{\wedge}_{I_n} \longleftrightarrows \dcat^{per}(\Mod_{E_0})^{\wedge}_{I_n} \colon \epsilon^*
\]
used in \Cref{prop:complete_module} and apply \Cref{prop:bs_module} with $T = (P(E_0E)^{\wedge}_{I_n})^{\htimes (k-1)}$. This then gives an adjunction 
\[
\xi_* \colon \Mod_{\dcat^{per}(\ComodE)^{\wedge}_{I_n}}(T)\longleftrightarrows \Mod_{\dcat^{per}(\Mod_{E_0})^{\wedge}_{I_n}}(\epsilon_*(T))\colon \xi^*
\]
and an equivalence 
\[
\Mod_{\dcat^{per}(\Mod_{E_0})^{\wedge}_{I_n}}(\epsilon_*(T)) \simeq \Mod_{\dcat^{per}(\ComodE)^{\wedge}_{I_n}}(\xi^*\epsilon_*(T)).
\]
But 
\[\epsilon_*(T) = \epsilon_*((P(E_0E)^{\wedge}_{I_n})^{\htimes (k-1)}) \simeq (P(E_0E)^{\wedge}_{I_n})^{\htimes (k-1)},\]
 as we just forget the (completed) comodule structure, and (recalling that $\xi^*$ is just the same as $\epsilon^*$ after forgetting the module structure)
 \[
 \begin{split}
\xi^*\epsilon_*(T) &= \xi^*\epsilon_*((P(E_0E)^{\wedge}_{I_n})^{\htimes (k-1)})\\
&=\epsilon^*((P(E_0E)^{\wedge}_{I_n})^{\htimes (k-1)})
\end{split}
 \]
 Using the definitions and the commutativity of the square \eqref{eq:commutative}, this is the same as $(P(E_0E)^{\wedge}_{I_n})^{\htimes k}$, as required. 
\end{proof}
\begin{Thm}\label{thm:picss_complete}
	Suppose $p > n+1$, then the commutative algebra object $P(E_0E)^{\wedge}_{I_n}$ is descendable in $\dcat^{per}(\ComodE)^{\wedge}_{I_n} $, and hence there is an equivalence of presentably symmetric monoidal stable $\infty$-categories 
	\[
  \resizebox{1 \textwidth}{!} {$\displaystyle
\dcat^{per}(\ComodE)^{\wedge}_{I_n} \simeq \Tot\left(\!\begin{tikzcd}[ampersand replacement=\&]
	{\dcat^{per}(\Mod_{E_0})^{\wedge}_{I_n}} \& \Mod_{\dcat^{per}(\Mod_{E_0})^{\wedge}_{I_n}}(P(E_0E)^{\wedge}_{I_n}) \& \cdots
	\arrow[shift left=1, from=1-2, to=1-3]
	\arrow[shift right=1, from=1-1, to=1-2]
	\arrow[shift left=1, from=1-1, to=1-2]
	\arrow[shift right=1, from=1-2, to=1-3]
	\arrow[from=1-2, to=1-3]
\end{tikzcd} \right)
$}
\]
\end{Thm}
\begin{proof}
 Because $P(E_0E)$ is descendable in $\dcat^{per}(\ComodE)$ \cite[Corollary 3.20]{Mathew2016Galois} gives that $P(E_0E)^{\wedge}_{I_n} \in \dcat^{per}(\ComodE)^{\wedge}_{I_n}$ is descendable. Applying \cite[Proposition 3.22]{Mathew2016Galois} we then have an equivalence between $\dcat^{per}(\ComodE)^{\wedge}_{I_n}$ and the totalization of 
\[
\resizebox{1 \textwidth}{!} 
{
\begin{tikzcd}[sep = small,ampersand replacement=\&] 
{\Mod_{\dcat^{per}(\ComodE)^{\wedge}_{I_n}}(P(E_0E)^{\wedge}_{I_n}}) \& {\Mod_{\dcat^{per}(\ComodE)^{\wedge}_{I_n}}((P(E_0E)^{\wedge}_{I_n})^{\htimes 2})} \& \cdots
	\arrow[shift right=1, from=1-1, to=1-2]
	\arrow[shift left=1, from=1-1, to=1-2]
	\arrow[from=1-2, to=1-3]
	\arrow[shift left=1, from=1-2, to=1-3]
	\arrow[shift right=1, from=1-2, to=1-3]
\end{tikzcd}
}
\]
Now apply \Cref{prop:complete_module} and \Cref{prop:descendable_module_complete}. 
\end{proof}
\begin{Cor}\label{cor:i_n_picard_tot}
  For $p > n+1$, then there is an equivalence of connective spectra between $\pics(\dcat^{per}(\Comod_{E_0E})^{\wedge}_{I_n}) $ and 
  \[
\tau_{\ge 0}\Tot\left(\adjustbox{scale=0.94}{%
\begin{tikzcd}
{\pics(\dcat^{per}(\Mod_{E_0})^{\wedge}_{I_n})} & {\pics(\Mod_{\dcat^{per}(\Mod_{E_0})^{\wedge}_{I_n}}(P(E_0E)^{\wedge}_{I_n}))} & \cdots
  \arrow[shift left=1, from=1-2, to=1-3]
  \arrow[shift right=1, from=1-1, to=1-2]
  \arrow[shift left=1, from=1-1, to=1-2]
  \arrow[shift right=1, from=1-2, to=1-3]
  \arrow[from=1-2, to=1-3]
\end{tikzcd}} \right).
\]

\end{Cor}
Before we can identify the $E_2$-term of the associated spectral sequence, we need to identify the relevant endomorphism rings. 
\begin{Lem}\label{lem:endo_1}
  The endomorphism ring of the unit
  \[
\pi_*\Hom_{\dcat^{per}(\Mod_{E_0})^{\wedge}_{I_n}}(P(E_0)^{\wedge}_{I_n},P(E_0)^{\wedge}_{I_n}) \cong E_*
  \]
\end{Lem}
\begin{proof}
  This follows by adjunction:
\[
\begin{split}
\pi_*\Hom_{\dcat^{per}(\Mod_{E_0})^{\wedge}_{I_n}}(P(E_0)^{\wedge}_{I_n},P(E_0)^{\wedge}_{I_n}) & \cong \pi_*\Hom_{\dcat^{per}(\Mod_{E_0})}(P(E_0),P(E_0)^{\wedge}_{I_n})  \\
 & \cong \pi_*\Hom_{\dcat(\Mod_{E_0})}(E_0,P(E_0)^{\wedge}_{I_n})\\
 & \cong H^*((E_*)^{\wedge}_{I_n})
\end{split}
\]
Because $E_*$ is a free $E_0$-module, \Cref{rem:complete_free} implies that the endomorphism ring is isomorphic to the $I_n$-adic completion of $E_*$, but this is already $I_n$-complete, and the result follows. 
\end{proof}
\begin{Lem}\label{lem:end_2}
Let $\cat C = \Mod_{\dcat^{per}(\Mod_{E_0})^{\wedge}_{I_n}}((P(E_0E)^{\wedge}_{I_n})^{\htimes k})$.  The endomorphism ring of the unit
  \[
\pi_*\Hom_{\cat C}(\unit,\unit) \cong (E_*E^{\otimes k})^{\wedge}_{I_n} \cong \Hom^c(\mathbb{G}_n^{\times k},E_*).   
  \]
\end{Lem}
\begin{proof}
  This is similar to the previous lemma: first, we note that by adjunction
\[
\begin{split}
\pi_*\Hom_{\cat C}(\unit,\unit) & \cong \pi_*\Hom_{\dcat^{per}(\Mod_{E_0})^{\wedge}_{I_n}}(P(E_0)^{\wedge}_{I_n},(P(E_0E)^{\wedge}_{I_n})^{\htimes k}) \\ & \cong \pi_*\Hom_{\dcat^{per}(\Mod_{E_0})}(P(E_0),P(E_0E)^{\wedge}_{I_n})^{\htimes k}) \\
& \cong \pi_*\Hom_{\dcat(\Mod_{E_0})}(E_0,P(E_0E)^{\wedge}_{I_n})^{\htimes k}) \\
& \cong H^*((E_*E)^{\wedge}_{I_n})^{\htimes k}). 
\end{split}
\]
Using that $\htimes$ is the symmetric monoidal structure, we have $$((E_*E)^{\wedge}_{I_n})^{\htimes k} \simeq (E_*E^{\otimes k})^{\wedge}_{I_n}$$. Once again we note that $E_*E^{\otimes k}$ is a flat $E_0$-module (as the tensor product of such modules), so that \Cref{rem:complete_free} gives that the endomorphism ring is isomorphic to $H_0^{I_n}(E_*E^{\otimes k}) \cong (E_*E^{\otimes k})^{\wedge}_{I_n}$ as claimed. 

In order to prove the last part of the lemma, let us temporarily introduce the notation $M \boxtimes N \coloneqq H_0^{I_n}(M \otimes N)$; this gives a symmetric monoidal structure on the category of $L$-complete $E_0$-modules, see \cite[Corollary A.7]{HoveyStrickl1999Morava}. Then, 
\[
H_0^{I_n}(E_*E^{\otimes k}) \cong E_*E^{\boxtimes k}
\]
To identify this further we note that $H_0^{I_n}(E_*E) \cong E^{\vee}_*E \coloneqq \pi_*L_{K(n)}(E \wedge E) \cong \Hom^c(\mathbb{G}_n,E_*)$. This is a combination of results; the first follows from the (collapsing) spectral sequence \cite[Theorem 2.3]{Hovey2008Morava}, while the second is then the main theorem of \cite{Hovey2004Operations}. Then,
\[
E_*E^{\boxtimes k} \cong E^{\vee}_*E^{\boxtimes k} \cong \Hom^c(\mathbb{G}_n,E_*)^{\boxtimes k} \cong \Hom^c(\mathbb{G}_n^{\times k},E_*),
\]
where the last step follows from the argument in the appendix of \cite{Devinatz1995Moravas}. 
\end{proof}

\begin{Thm}\label{cor:kn_ss}
Suppose $p > n+1$, then there is a spectral sequence
	\[
E_2^{s,t} \cong \begin{cases}
	\mathbb{Z}/2 & s = t = 0 \\
	H^s(\mathbb{G}_n,E_0^{\times}) & t = 1\\
H^s(\mathbb{G}_n,E_{t-1}) & t \ge 2. 
\end{cases}
	\]
	which converges for $t -s \ge 0$ to $\pi_{t-s}\pics(\dcat^{per}(\ComodE)^{\wedge}_{I_n})$. The differentials run $d_r \colon E_r^{s,t} \to E_r^{s+r,t+r-1}$. 
\end{Thm}
\begin{proof}
This is the Bousfield--Kan spectral sequence associated to the tower of \Cref{cor:i_n_picard_tot}.  When $s = t = 0$, the identification of the $E_2$-term is a computation of $\Pic(\dcat^{per}(\Comod_{E_0E})^{\wedge}_{I_n})$, for which see \Cref{lem:local_pic}. For $t \ge 2$, we use \Cref{lem:endo_1,lem:end_2} and the definition of the Picard spectrum (in particular, \eqref{pic:defn}), to compute that the $E_2$-term in this range is the cohomology of the complex
\[\begin{tikzcd}
  {E_{t-1}} & {\Hom^c(\mathbb{G}_n,E_{t-1})} & {\Hom^c(\mathbb{G}_n \times \mathbb{G}_n,E_{t-1})} & \cdots
  \arrow[from=1-1, to=1-2]
  \arrow[shift left=1, from=1-2, to=1-3]
  \arrow[shift right=1, from=1-2, to=1-3]
  \arrow[shift left=1, from=1-3, to=1-4]
  \arrow[shift right=1, from=1-3, to=1-4]
  \arrow[from=1-3, to=1-4]
\end{tikzcd}\]
which is the standard cobar complex computing $H^*(\mathbb{G}_n,E_{t-1})$. The argument for $t = 1$ is similar; one just needs to note that the units of $\Hom^c(\mathbb{G}_n^{\times (k-1)},E_0)$ are exactly $\Hom^c(\mathbb{G}_n^{\times (k-1)},E_0^{\times})$. 
\end{proof}
\begin{Thm}\label{thm:main_pic_local}
	Suppose that $2p-2 \ge n^2$ and $(p-1) \nmid n$, then there is a short exact sequence 
	\[
0 \to H^1(\mathbb{G}_n,E_0^{\times}) \to \Pic(\dcat^{per}(\ComodE)^{\wedge}_{I_n}) \to \mathbb{Z}/2 \to 0.
	\]\end{Thm}
	\begin{proof}
		The condition that $2p-2 \ge n^2$ implies that $p > n+1$ except when $n = 1, p=2$ but this case is ruled out by the condition that $p-1 \nmid n$. In other words, the conditions of the theorem imply that $p > n+1$ always holds, so that \Cref{thm:picss_complete,cor:kn_ss} apply. The condition that $(p-1) \nmid n$ implies that the spectral sequence has a horizontal vanishing line above $s = n^2$ \cite[Theorem 6.2.10]{Ravenel1986Complex}. Moreover, a standard sparseness argument shows that $H^s(\mathbb{G}_n,E_t) = 0$ unless $t$ is divisible by $2p-2$. Indeed, $p$ is always odd, and we consider the Hochschild--Serre spectral sequence associated to the central subgroup $\mathbb{Z}_p^{\times} \simeq \mu_{p-1} \oplus (1+p\mathbb{Z}_p)^{\times} \subseteq \mathbb{G}_n$, i.e., the spectral sequence 
		\[
H^p(\mathbb{G}_n/\mu_{p-1},H^q(\mu_{p-1},E_t)) \implies H^{p+q}(\mathbb{G}_n,E_t). 
		\]
		By considering the standard resolution, we have that $H^q(\mu_{p-1},E_t) = 0$ unless $q = 0$ and $t$ is a multiple of $2(p-1)$. 

		Combined, these two conditions imply that the only contributing terms in the $t - s = 0$ column are in filtration degree 0 and 1, and these are not involved in any differentials. The statement of the theorem follows. 
	\end{proof}
	\begin{Rem}
		The same result holds for the $K(n)$-local Picard group. In other words, the $K(n)$-local Picard group and the Picard group of $\cat{D}^{per}(\ComodE)^{\wedge}_{I_n}$ agree, at least up to extension, when $2p-2 \ge n^2$ and $(p-1) \nmid n$. 
	\end{Rem}
  Using known computations, we can be more specific in certain cases. 
	\begin{Thm}\label{thm:height1_kn}
		When $n = 1$ and $p > 2$
		\[
\Pic(\dcat^{per}(\ComodE)^{\wedge}_{I_1}) \cong \mathbb{Z}_p \times \mathbb{Z}/2(p-1). 
		\]
		\end{Thm}
		\begin{proof}
Here we have $\mathbb{G}_1 \cong \bbZ_p^{\times}$, so that there is a short exact sequence 
		\[
0 \to H^1(\mathbb{Z}_p^{\times},\mathbb{Z}_p^{\times}) \to \Pic(\dcat^{per}(\ComodE)^{\wedge}_{I_n}) \to \mathbb{Z}/2 \to 0. 
		\]
		The group $H^1(\mathbb{Z}_p^{\times},\mathbb{Z}_p^{\times}) \cong \mathbb{Z}_p^{\times}$ is topologically generated by the crossed homomorphism 
		\[
		\begin{split}
t_0 \colon \mathbb{Z}_p^{\times} &\to (E_1)_0^{\times}\\
g &\mapsto g_*(u)/u
\end{split}
		\]
		Because $u$ has degree 2, this corresponds to $P(E_0)[2] \cong E_{*+2}$. Thus, this class generates a copy of $\mathbb{Z}_p^{\times} \cong \mathbb{Z}_p \times \mathbb{Z}/(p-1)$ in $\Pic(\dcat^{per}(\ComodE)^{\wedge}_{I_1})$. But $P(E_*)[1]$ is not an element of the kernel, and therefore its image is a generator of $\mathbb{Z}/2$ (we can also see this directly from the spectral sequence). We see that $P(E_0)[1] \cong E_*[1]$ generates $\mathbb{Z}_p \times \mathbb{Z}/2(p-1)$, as claimed. 
	\end{proof}
	\begin{Thm}\label{thm:pickn_ht2}
		When $n = 2$ and $p > 3$, then 
		\[
\Pic(\dcat^{per}(\ComodE)^{\wedge}_{I_2}) \cong \mathbb{Z}_p^2 \times \mathbb{Z}/2(p^2-1). 
		\]
	\end{Thm}
  \begin{proof}
    The computation of $H^*(\mathbb{G}_2,E_0^{\times})$ (due to Hopkins) can be found in \cite[Theorem 8.1]{Behrens2012homotopy} (or \cite[Theorem 5.22]{lader2013resolution}); in particular, we have
    \[
    H^*(\mathbb{G}_2,E_0^{\times}) \cong \bbZ_p^2 \times \bbZ/(p^2-1)
    \]
    This is generated topologically by the crossed homomorphisms 
        \[
    \begin{split}
t_0 \colon \mathbb{G}_2 &\to E_0^{\times}\\
g &\mapsto g_*(u)/u
\end{split}
\]
    and the determinant element $\det$ given as the homomorphism
    \[
\mathbb{G}_2 \xrightarrow{\det} \Z_p^{\times} \xrightarrow{\subseteq} E_0^{\times}. 
    \] 
    As in \Cref{thm:height1_kn}, $t_0$ corresponds to $P(E_0)[2] \cong E_{*+2}$, and the extension
    \[
0 \to \bbZ_p^2 \times \bbZ/(p^2-1) \to \       
\Pic(\dcat^{per}(\ComodE)^{\wedge}_{I_2}) \to \bbZ/2 \to 0
    \]
    does not split, giving the result. 
  \end{proof}
\appendix
\section{The monoidal Barr--Beck theorem}\label{sec:monoida_barr}
In this appendix we review the monoidal Barr--Beck theorem due to Mathew--Naumann--Noel \cite[Proposition 5.29]{MathewNaumannNoel2017Nilpotence} (inspired by earlier work of Balmer--Dell'Ambrogio--Sanders \cite{BalmerDellAmbrogioSers2015Restriction}), and some extensions of it due to Behrens and Shah \cite{BehrensShah2020equivariant}. We claim no originality for these results; we simply include them for the convenience of the reader. 

\begin{Thm}[Mathew--Naumann--Noel]\label{thm:mnn_bb}
  Let $F \colon \cat C \longleftrightarrows \dcat \colon G$ be a symmetric monoidal adjunction between presentably stable $\infty$-categories. Suppose that the adjunction satisfies the following conditions:
  \begin{enumerate}
  \item $G$ is conservative.  
  \item $G$ commutes with colimits. 
  \item The adjunction satisfies the projection formula: the natural map
  \[
G(X) \otimes Y \to G(X \otimes F(Y))
  \]
  is an equivalence for all $X \in \dcat$ and $Y \in \cat C$. 
\end{enumerate}
  Then there is an equivalence of symmetric monoidal presentably stable  $\infty$-categories $\dcat \simeq \Mod_{\cat C}(G(\unit_{\dcat}))$ under which the adjunction $(F,G)$ is equivalent to the extension of scalars/restriction adjunction along the morphism of commutative algebra objects $\unit_{\cat C} \to G(\unit_{\dcat})$. 
\end{Thm}
\begin{Rem}\label{rem:localized_adjunction}
  Given an adjunction $F \colon \cat C \longleftrightarrows \dcat \colon G$ and an arbitrary $T \in \cat C$, there is an induced symmetric monoidal adjunction 
  \[
F' \colon L_T\cat C \longleftrightarrows L_{F(T)}\dcat \colon G'
  \]
Here $G'$ is the restriction of $G$ and $F' = L_{F(T)}F$. See \cite[Lemma 3.7]{BehrensShah2020equivariant}, which also proves the following. 
\end{Rem}
\begin{Prop}[Behrens--Shah]\label{prop:bs_complete}
  Let $F \colon \cat C \longleftrightarrows \dcat \colon G$ be as in \Cref{thm:mnn_bb}, and let $T \in \cat C$. Then the induced symmetric monoidal adjunction of \Cref{rem:localized_adjunction}
  \[
F' \colon L_T\cat C \longleftrightarrows L_{F(T)}\dcat \colon G'
  \]
  satisfies the conditions of \Cref{thm:mnn_bb}, and there is an equivalence $L_TG \simeq G'L_{F(T)}$, i.e., the diagram
\begin{equation}\begin{tikzcd}
  {\cat C} & {\cat D} \\
  {L_T\cat C} & {L_{F(T)}\cat D}
  \arrow["G"', from=1-2, to=1-1]
  \arrow["{L_{F(T)}}", shift right=1, from=1-2, to=2-2]
  \arrow["{G'}", from=2-2, to=2-1]
  \arrow["{L_T}"', shift right=1, from=1-1, to=2-1]
\end{tikzcd}\end{equation}
commutes. In particular, there are equivalences of  symmetric monoidal presentably stable $\infty$-categories 
  \[
L_{F(T)}\dcat \simeq L_{F(T)}\Mod_{\cat C}(G(\unit_{\dcat})) \simeq \Mod_{ L_T\cat C}(L_TG(\unit_{\dcat})).
  \]
\end{Prop}

\begin{Rem}\label{rem:localized_module}
  Given an adjunction $F \colon \cat C \longleftrightarrows \dcat \colon G$ and a commutative algebra object $T \in \CAlg(\cat C)$, there is an induced symmetric monoidal adjunction 
  \[
F' \colon \Mod_{\cat C}(T) \longleftrightarrows \Mod_{\cat D}(F(T)) \colon G',
  \]
where $F'$ and $G'$ are computed by $F$ and $G$ after forgetting the module structure. The following is \cite[Lemma 3.9]{BehrensShah2020equivariant}.
\end{Rem}
\begin{Prop}[Behrens--Shah]\label{prop:bs_module}
  Let $F \colon \cat C \longleftrightarrows \dcat \colon G$ be as in \Cref{thm:mnn_bb}, and let $T \in \CAlg(\cat C)$. Then the induced symmetric monoidal adjunction 
  \[
F' \colon \Mod_{\cat C}(T) \longleftrightarrows \Mod_{\dcat}(F(T)) \colon G'
  \]
  satisfies the conditions of \Cref{thm:mnn_bb}. In particular, there are equivalences of presentably symmetric monoidal $\infty$-categories 
  \[
\Mod_{\dcat}(F(T)) \simeq \Mod_{\Mod_{\cat C}(T)}(G'(F(T))) \simeq \Mod_{\cat C}(G'(F(T))). 
  \]
\end{Prop}
\begin{Rem}\label{Rem:commutative_diagram_modules}
  Let $\alpha^* \colon \cat C \leftrightarrows \Mod_{\cat C}(T) \colon \alpha_*$ denote the extension of scalars/restriction adjunction along $\unit_{\cat C} \to T$. We note that the diagram
\[\begin{tikzcd}
    {\cat C} & {\cat D} \\
    {\Mod_{\cat C}(A)} & {\Mod_{\cat D}(F(A))}
    \arrow["F", shift left=1, from=1-1, to=1-2]
    \arrow["G", shift left=1, from=1-2, to=1-1]
    \arrow["{F'}", shift left=1, from=2-1, to=2-2]
    \arrow["{G'}", shift left=1, from=2-2, to=2-1]
    \arrow["{\alpha^*}"', shift right=1, from=1-1, to=2-1]
    \arrow["{\alpha_*}"', shift right=1, from=2-1, to=1-1]
    \arrow["{\beta^*}"', shift right=1, from=1-2, to=2-2]
    \arrow["{\beta_*}"', shift right=1, from=2-2, to=1-2]
\end{tikzcd}\]
commutes, up to natural equivalence, in the sense that
\[ 
\begin{tikzcd}[row sep=tiny]
F' \circ \alpha^* \simeq \beta^* \circ F & G' \circ \beta^* \simeq \alpha^* \circ G \\
\beta_* \circ F' \simeq F \circ \alpha_* & G \circ \beta_* \simeq \alpha_* \circ G'.
\end{tikzcd}
\]
Indeed, $F' \circ \alpha^* \simeq \beta^* \circ F $ is the claim that $F(T) \otimes F(M) \simeq F(T \otimes M)$ for $M \in \cat C$, which follows because $F$ is symmetric monoidal. Taking right adjoints we see that $G \circ \beta_* \simeq \alpha_* \circ G$. The statement $G' \circ \beta^* \simeq \alpha^* \circ G$ is that claim that $G(F(T) \otimes N) \simeq T \otimes G(N)$ for $N \in \cat D$, which follows from the projection formula, while $\beta_* \circ F' \simeq F \circ \alpha_*$ follows from that fact that $F'$ is defined as $F$ after forgetting the module structure. 
\end{Rem}
\bibliographystyle{alpha}
\bibliography{stable} 
\end{document}